\documentclass[11pt]{article}
\usepackage[a4paper, total={6in, 8in}]{geometry}
\usepackage{eurosym}
\usepackage{color}
\usepackage{amsmath}
\usepackage{amsfonts}
\usepackage{amssymb}
\usepackage{authblk}
\usepackage{amsthm}
\usepackage{dsfont}
\usepackage{mathtools}
\usepackage{scalerel,nicefrac}
\usepackage{tikz}
\usepackage{adjustbox}
\usepackage{bm}

\newtheorem{theorem}{Theorem}[section]
\theoremstyle{plain}
\newtheorem{lem}[theorem]{Lemma}
\newtheorem{cor}[theorem]{Corollary}
\theoremstyle{definition}
\newtheorem{rem}{Remark}
\newtheorem{dfn}{Definition}

\newtheorem*{con}{Convention}
\newtheorem{q}{Question}

\newcommand{\M}{\mathcal{M}}
\newcommand{\N}{\mathcal{N}}
\newcommand{\K}{\mathcal{K}}
\newcommand{\LL}{\mathcal{L}}
\newcommand{\NN}{\mathbb{N}}
\newcommand{\I}{\mathrm{I_{fix}}}
\newcommand{\G}{\mathrm{G}}
\newcommand{\E}{\mathrm{E}}
\newcommand{\A}{\mathrm{Aut}}
\newcommand{\Aa}{\A_{(I)}^{\M}(\K)}
\newcommand{\pa}{\mathrm{PA}}
\newcommand{\ssy}{\mathrm{SSy}}

\newcommand{\fix}{\mathrm{Fix}}
\newcommand{\LG}{\mathrm{LG}}

\newcommand{\KK}{\mathrm{K}}

\newcommand{\tp}{\mathrm{tp}}

\makeatletter
\def\moverlay{\mathpalette\mov@rlay}
\def\mov@rlay#1#2{\leavevmode\vtop{%
		\baselineskip\z@skip \lineskiplimit-\maxdimen
		\ialign{\hfil$\m@th#1##$\hfil\cr#2\crcr}}}
\newcommand{\charfusion}[3][\mathord]{
	#1{\ifx#1\mathop\vphantom{#2}\fi
		\mathpalette\mov@rlay{#2\cr#3}
	}
	\ifx#1\mathop\expandafter\displaylimits\fi}
\makeatother

\newcommand{\bigcupdot}{\charfusion[\mathop]{\bigcup}{\cdot}}

\begin{document}

\title{The automorphism group of countable recursively saturated models of Peano arithmetic and strong cuts}
\author{Saeideh Bahrami\footnote{
		School of Mathematics, Institute for Research in Fundamental Sciences (IPM), P.O. Box: 19395-5764, Tehran, Iran.\\
		This research was  supported by a grant from Iran National Science Foundation (INSF) (NO.99019371).} \\
	
	{\small {bahrami.saeideh@gmail.com}}}
\maketitle

\begin{abstract}
In this paper, we extend the concept of a Lascar generic automorphism in the setting of models of Peano arithmetic ($\pa$) to the subgroup of the automorphism group of a countable recursively saturated model $\M$ of $\pa$ that fixes pointwise a strong cut $I$ of $\M$, denoted by $(\A(\M))_{(I)}$. Then, we prove that:
\begin{itemize}
\item[(1)] $(\A(\M))_{(I)}$ has the small index property.
\item[(2)] The cofinality of $(\A(\M))_{(I)}$ is uncountable.
\item[(3)] Any nontrivial normal subgroup of  $(\A(\M))_{(I)}$ is meagre in it. In particular, the infinite cyclic group $\mathbb{Z}$ is not a homomorphic image of $(\A(\M))_{(I)}$.
\end{itemize} 
\end{abstract}

\section{Introduction}
In  early 1991, J. K. Truss introduced \cite{trusss} the concept of a \textit{generic}\footnote{The term "generic" is used because of the close connection between this notion and the ideas of forcing in set theory.} automorphism, aiming to identify an  element of a permutation group which 'almost' intersects with  all the finite partial isomorphisms of the underlying set. To be more precise, if $\G$ is a permutation group endowed with a Polish topological space (i.e., a  separable complete metric space),    an element $g$ of $\G$ is called \textit{generic} if its conjugacy class is \textit{comeagre} in  $\G$---that is, its conjugacy class is equal to the intersection of a countable family of open dense subsets of $\G$.  It is well-established that  the  automorphism group of any structure $ \M $, denoted by $ \A(\M) $, forms a topological group. Furthermore, when $ \M $ is countable, $ \A(\M) $ is a Polish group  \cite{km}. In his paper \cite{trusss}, Truss explored the existence and the properties of generic automorphisms of a number of homogeneous structures, such as a countable set with its symmetric group, $ (\mathbb{Q},<) $ which is the set of rational numbers with its ordering, and a countable universal $ C $-coloured graph. 

Concurrently, D. Lascar examined the automorphism group of various  structures in a series of papers (\cite{las1}, \cite{las2}, and \cite{las3}), and came out with a notion very similar to the concept of a generic automorphism introduced by Truss. Lascar’s notion is based on strong \textit{amalgamation} properties and plays an important role in showing the \textit{small index property} in a variety of structures (e.g., see \cite{hod} or \cite{las-sh}); where a topological group of cardinality $\kappa$ is said to have the \textit{small index property} if every subgroup whose index is less than $\kappa$ is open in the ground space.  So, the small index property guarantees that the topological  structure of a topological group can be reduced to its abstract group structure. Moreover, it is intimately connected to the extent to which a structure can be recovered from its automorphism groups. In $\cite{las-sh}$, Lascar and S. Shelah  showed that the automorphism group of any uncountable saturated structure have the small index property. Then in \cite{las}, Lascar modified their method to show that the automorphism group of every countable \textit{arithmetically saturated} model of Peano arithmetic ($\pa$)  has the small index property; where an \textit{arithmetically saturated} model of $\pa$ is a recursively saturated model in which the standard cut is \textit{strong} (for precise definitions and details, see the Preliminaries section below).

In section 3 of this paper, we first generalize the method Lascar used in \cite{las}. Then, we show that the Lascar's result on countable arithmetically saturated models of $\pa$ can be generalized to every automorphism group of a countable recursively saturated model of $\pa$ which pointwise stabilizes a strong cut of the ground model:
\newtheorem*{aaaa}{Theorem A}
\begin{aaaa}
Suppose $\M$ is a countable recursively saturated model of $\pa$ and $I$ is a strong cut of $\M$. Then the subgroup of $\A(\M)$ which fixes  $I$ pointwise, denoted by $(\A(\M))_{(I)}$,  has the small index property.
\end{aaaa}

Section 4 of this paper is devoted to some  observations  concerning the introduced  notion of a Lascar generic automorphism for $(\A(\M))_{(I)}$, which we call a Lascar $I$-generic automorphism. We begin by examining the set of fixed points of a Lascar $I$-generic automorphism. Then, we generalize a theorem by R. Kaye \cite{k} which demonstrates a Galois-like correspondence between closed normal subgroups of the automorphism group of a countable recursively saturated model of $\pa$ and its \textit{invariant} cuts which are closed under exponentiation.  We show that if $I$ is a strong cut of $\M $, then  Kaye's result also holds for   $(\A(\M))_{(I)}$. As a consequence, we use this generalization along with the notion of Lascar $I$-generic automorphisms to derive a property for the normal subgroups of $(\A(\M))_{(I)}$:
\newtheorem*{bb}{Theorem B}
\begin{bb}
	Suppose $\M$ is a countable recursively saturated model of $\pa$ and $I$ is a strong cut of $\M$. Then for every  normal subgroup $\mathrm{N}$ of $(\A(\M))_{(I)}$ either $(\A(\M))_{(I)}=\mathrm{N} $ or $\mathrm{N} $ is  meagre in $(\A(\M))_{(I)}$.
	\end{bb}

In section 5, we investigate how $(\A(\M))_{(I)}$ can be covered by a chain of its proper subgroups. In other words, we compute the cofinality of $(\A(\M))_{(I)}$; where
 the \textit{cofinality} of a group $\G$ is the least cardinality of a well-ordered chain of proper subgroups whose union is $\G$. By generalizing a result of R. Kosssak and J. Schmerl \cite{ks}, we show that:
\newtheorem*{cc}{Theorem C}
\begin{cc}
	Suppose $\M$ is a countable recursively saturated model of $\pa$ and $I$ is a  cut of $\M$ which is not $\omega$-coded from above. Then $I$ is strong in $\M$ iff $ (\A(\M))_{(I)}$ has uncountable cofinality.
\end{cc}

\section{Preliminaries}
Throughout this paper, we mostly work in the language of Peano arithmetic, that is ${\mathbb{L}_{A}:=\{0,1,+,.,<\}}$. For every language $\mathbb{L}$ extending $\mathbb{L}_{A}$, we denote $\pa(\mathbb{L})$ by  $\pa^{*}$. Moreover, we use $\M$, $\K$, $\LL$,  and similar notations to refer to $\mathbb{L}$-structures (for a given language $\mathbb{L}$) where the universes  of these structures are represented by $M$, $K$, $L$, and etc.., respectively. 
 A given $\mathbb{L}$-structure $ \M $ is called  \textit{recursively saturated} if it realizes every recursive type with a finite number of parameters from  $ M $.  In \cite{bs}, J. Barwise and J. Schlipf showed that  \textit{any countable  model $ \M $ of $ \pa $ is recursively saturated iff it carries an \textit{inductive satisfaction class}}; here  an \textit{inductive satisfaction class} $  S $ of $ \M $ is a subset of $ M $ which contains pairs $ \langle\varphi,a\rangle $ such that (1) $ \M\models \mathrm{Form}(\varphi) $, (2) $ (\M; S)\models\pa^{*} $, and (3) $ (\M; S) $ satisfies Tarski's inductive conditions for satisfaction (for a more precise definition see \cite{kaye}).  In \cite{sm},  C. Smory\'{n}ski by generalizing   Barwise-Ressayre  expandability result, proved that   \textit{for every countable recursively saturated model $ \M $ of $ \pa $ there exists some inductive satisfaction class $  S $ such that $ (\M; S) $ is also recursively saturated}.  Moreover, For every model $ \M $ of $\pa$, we define the \textit{standard system} of $ \M $ as below:
 $$ \ssy(\M):=\{X\cap\NN: \ X \text{ is definable in } \M \}. $$ 
  \begin{theorem}[Smory\'{n}ski \cite{sm}]
 	Every two recursively saturated models of $ \pa $ are isomorphic iff they have the same theories and the same standard systems.
 \end{theorem}
A model  $ \M $ is called  \textit{arithmetically saturated} if for every finite tuple $ \bar{a}\in M $, each arithmetic type in $ \tp(\bar{a}) $  is realizable in $ \M $. It is well-known that \textit{a recursively saturated model $ \M $ of $ \pa $ is arithmetically saturated iff the standard cut $ \NN $ is strong in $ \M $} (e.g., see \cite{wil}); where a \textit{strong} cut is defined as follows: a  cut $ I $ of a  model $ \M $ is called \textit{strong} if for every coded function $ f $ in $ \M $ whose domain contains $ I $ there exists some $ e>I $ such that $ f(i)\in I $ iff $ f(i)<e $ for all $ i\in I $. 
  
  For every $ X\subseteq M $, by $\KK(\M;X)$ we mean the set of all elements of $M$ which are definable in $\M$ with parameters from $X$. If $X=\{a\}$ for some $a\in M$, we simply denote $\KK(\M;X)$ by $\KK(\M;a)$. 

For every $\mathbb{L}$-formula $\phi(x) $ (where $\mathbb{L}$ is an extension of $\mathbb{L}_{A}$), we have the following $\mathbb{L}$-formula:
\begin{center}
	$y=\mu_{x} \ \phi(x): \ \ $ \textit{'$y$ is the least element such that $\phi(y)$ holds'}.
\end{center}
Moreover, the following  $\mathbb{L}_{A}$-formulas are available within $\pa$:
\begin{center}
\begin{itemize}
\item $x\E y: \ \ $ \textit{'$x$ is an Ackermann's member of $y$'}; in other words, \textit{'the $x$-th bit of the binary expansion of $y$ is 1'}. We might also say $x$ is an $\E$-member of $y$.
\item $(x)_{\E}=y: \ \ $ \textit{'$y$ is the code of the set of $\E$-members of $x$'}.
\item $x^{\frown}z=y: \ \ $ \textit{'$y$ is the sequence number obtained by adding $z$ at the end of the sequence coded by $x$'}.
\item $(x)_{y}=z: \ \ $\textit{'the $y$-th element of the sequence coded by $x$ is $z$'}.
\item $\mathrm{Card}(x)=y: \ \ $ \textit{'there exists a bijection between $y$ and the set coded by $x$'}.
\item 	$2_{0}^{x}:=x$ and
$2_{y+1}^{x}:=2^{2_{y}^{x}}$.
\item $\log^{0}x:=x$ and $\log^{y+1}x:=\log_{2}(log^{y}x+1)$.
\end{itemize}
\end{center} 
Suppose $I$ is a given cut of $\M$. A subset $X$ of $\M$ is called \textit{$I$-small} if $X=\{(a)_{i}: \ i\in I\}$ for some $a\in M$ such that $(a)_{i}\neq(a)_{j}$ for all distinct $i,j\in I$. If $I$ is the the standard cut of $\M$, we simply say $X$ is \textit{small} in $\M$.  Moreover,
we write $ \K\prec_{s}\M $ to indicate  that $ \K $ is a small elementary submodel of $ \M $.  Let $ \mathcal{S}(\M) $ and $\mathcal{S}_{I}(\M)$ be the family of small and $I$-small elementary submodels of $\M$, respectively; when there is no risk of ambiguity, we may simply write $\mathcal{S}$ and $\mathcal{S}_{I}$. Clearly,  $\mathcal{S} $ and $\mathcal{S}_{I}$ are countable, when $\M$ is countable. Any $I$-small elementary submodel of $\M$ may itself admit nontrivial automorphisms; thus,  for $I$-small elementary submodel  $ \K $ of $\M$, we  denote the family of automorphisms of $ \K $ which  can be extended to some automorphism of $ \M $ which stabilizes $I$ pointwise by $ \Aa $. Note that when $\M$ is countable, $\Aa$ is also countable. The next lemma summarizes some fundamental properties of $I$-small elementary submodels which we use in this paper:
\begin{lem}
	Suppose $\M$ is a recursively saturated model of $\pa$, $I$ is a strong cut of $\M$, and $\K=\{(a)_{i}: \ i\in I\}$ is an $I$-small elementary submodel of $\M$. Then:
	\begin{itemize}
		\item[(1)] \textup{(Bahrami \cite{me})}  $\KK(\M;I\cup\{\bar{b}\})$ is $I$-small in $\M$, for every finite tuple $\bar{b}\in M$.  
		\item[(2)] \textup{(Essentially Enayat \cite{ali})} $I$ is a subset of $K$.
		\item[(3)] \textup{(Bahrami \cite{me})} If $X$ is a subset of $K$ which is definable in $\M$ and there exists some $i_{0}\in I$ such that $i<i_{0}$ for every $(a)_{i}\in X$, then $X$ is definable in $\K$.
	\end{itemize}
\end{lem} 
\begin{rem}
Although parts (1) and (3) of Lemma 2.2 are stated and proved in \cite{me} for models of  $\mathrm{I}\Sigma_{1}$, they can be adapted for recursively saturated  models of  $\pa$. 
\end{rem}

Every countable recursively saturated model $ \M $ of $ \pa $ has a rich class of automorphisms, denoted by $ \A(\M) $. 
In this paper, we usually refer to $\A(\M)$ by $\G$ for simplicity. $\G$ forms a topological group whose  basic open sets are cosets of subsets of the form ${\G_{\bar{a}}:=\{g\in\G: \ g(\bar{a})=\bar{a} \}}  $, where $ \bar{a}$ is a finite tuple of elements of $ M $. In particular, $\G$ is a Polish group (see \cite{km} for details). For each $g,h\in\G$, by $[g]^{\G}$ we mean the conjugacy class of $g$ in $\G$, and $g^{h}:=h^{-1}gh$. For every automorphism $ g $ of $ \M $, let $ \I(g) $ be the largest initial segment of $ \M $ whose elements stay fixed by $ g $, and $ \fix(g) $ is the set of all fixed points of $ g $.  In \cite{kkk} it is shown that if $I$ is a strong cut of $\M$, then $\KK(\M;I)$ can appear as the fixed point set of some automorphism of $\M$:
\begin{theorem}[Kaye-Kossak-Kotlarski \cite{kkk}]
Suppose $\M$ is a countable recursively saturated model of $\pa$. Then $\NN$ is strong in $\M$ iff there exists some automorphism $g$ of $\M$ such that $\fix(g)=\KK(\M)$. In particular, J. Schmerl has noted in \cite[Theorem 5.7]{kkk} that the same proof can be adapted to show that: a cut $I$ of $\M$ is strong in $\M$ iff for every $I$-small elementary submodel $\K$ of $\M$ there exists some automorphism  $g$ of $\M$ such that $\fix(g)=K$.
\end{theorem}
A given cut $ I $ is called \textit{invariant} if $ g(I)=I $ for all $ g\in\G $. It is shown in \cite{kkk} that \textit{if $\M$ is a countable recursively saturated model of $\pa$, then a cut $I$ of $\M$ is invariant iff $\KK(\M)\cap I$ is cofinal in $I$ or $\KK(\M)\cap (M\setminus I)$ is downward cofinal in $M\setminus I$.}
For every subgroup $ \mathrm{H} $ of $ \G $, let $ {\I(\mathrm{H}):=\bigcap_{g\in \mathrm{H}} \I(g)} $. It is easy to see that  \textit{if $ H $ is normal in $ \G $, then $ \I(H) $ is invariant}.\\
 For every subset $X$ of $M$, $\G_{(X)}$ is the subgroup of $\G$ which fixes $X$ pointwise. Moreover, for a given cut $ I $ of $ \M $, we define:
\begin{center}
$ \G_{(>I)}:=\{g\in\G: \ I\subsetneqq\I(g) \}. $
\end{center}
It is easy to see that when $I$ is an invariant cut of $\M$, then both $\G_{(I)}$ and $\G_{(>I)}$ are normal subgroups of $\G$. Moreover, Kaye has proved that every closed normal subgroup of $\G$ is of the form $\G_{(I)}$  for some invariant cut $I$ of $\M$ which is closed under exponentiation:
\begin{theorem}[Kaye \cite{k}]
	Suppose $ \M $ is a countable recursively saturated model of $ \pa $ and $\G:=\A(\M)$. Then:
		\begin{itemize}
			\item[(1)] If $I=\inf \{\log^{n}(a): \ n\in \omega\} $ for some $a>I$, then  $\G_{(>I)}$ is equal to  $\G_{(I')} $, where $ I':=\sup \{2_{n}^{a}: \ n\in \omega\}$\footnote{For every subset $X$ of $M$ we have	$\sup(X):=\{a\in M: \ a\leq x \text{ for some  } x\in X\}$, and $\inf(X):=\{a\in M: \ a< x \text{ for all  } x\in X\}$.}. If $I$ is a cut of $\M$  not of the above form, then the closure of $\G_{(>I)}$  equals to $\G_{(I)} $.
		\item[(2)] For every closed normal subgroup $\mathrm{N}$ of $\G$, it holds that $\mathrm{N}=\G_{(\I(\mathrm{N}))}$.
	\end{itemize} 
\end{theorem}
Theorem 2.4 implies that the nontrivial closed normal subgroups (if they exists) are linearly ordered. Moreover, since every open subgroup of $\G$ is also closed, from the Kaye's Theorem we infer  that nontrivial normal subgroups are nonopen.  In \cite{k}, Kaye also conjectured  that  every nonclosed normal subgroup is of the form $\G_{(>I)}$ where $I$ is an invariant cut of $\M$ which is closed under exponentiation; this conjecture remains unproven.

As mentioned in the Introduction section, we  say  a  topological group $ \G $ of cardinality $2^{\aleph_{0}}$ has the \textit{small index property} (abbreviated by the \textit{SIP}) if every subgroup $ \mathrm{H} $ of $\G$ whose index is less than $ 2^{\aleph_{0}} $ is open in $ \G $. In \cite{las}, Lascar introduced a version of the notion of generic automorphism---now commonly referred to as a \textit{Lascar generic automorphism} in the literature on Peano arithmetic (Definition 2 below)---and used it to prove that:
\begin{theorem}[Lascar \cite{las}]
	The automorphism group of every countable arithmetically saturated model of $ \pa $ has the small index property.
\end{theorem}
In formulating his notion of a generic automorphism, Lascar employed the following concept of an \textit{existentially closed} automorphism: 
\begin{dfn}
$(n\in\omega)$	Suppose $ \K\prec_{s}\M $ and $ (g_{0},...,g_{n})\in(\A_{\NN}^{\M}(\K))^{n+1} $. We  say $ (g_{0},...,g_{n}) $ is \textit{existentially closed} (abbreviated by \textit{e.c.}) if for every  tuple  $ {(f_{0},...,f_{n})\in(\G)^{n+1}} $ such that  $ g_{k}= f_{k}\upharpoonright_{K} $ for all $k=0,...,n$,  for every $\mathbb{L}_{A}$-formula $ \varphi(x,y,\bar{z}) $, and for each $ \bar{c}\in K $ if $ \M\models\exists x \ \varphi(x,f_{0}(x),...,f_{n}(x),\bar{c}) $, then $ \M\models\varphi(d,g_{0}(d),...,g_{n}(d),\bar{c}) $ for some $ d\in K $.
\end{dfn}

\begin{theorem}[Lascar \cite{las}]
$(n\in\omega)$ Suppose $ \M $ is a countable  arithmetically saturated model of $ \pa $ and $ a,b_{0},...,b_{n}\in M $ such that $ \tp(a)=\tp(b_{k}) $ for all $k=0,...,n$. Then there exist some $ \K\prec_{s}\M $ containing $ a,b_{0},...,b_{k} $ and some e.c. tuple $ (g_{0},...,g_{n})\in(\A_{\NN}^{\M}(\K))^{n+1} $ such that  $ g_{k}(a)=b_{k} $ for all $k=0,...,n$.
\end{theorem}

\begin{dfn}
($n\in\omega$)	A tuple $ (g_{0},...,g_{n})\in(\G)^{n+1} $ is called \textit{Lascar generic} if:
	\begin{itemize}
		\item[(i)] For every finite tuple $ \bar{a}\in M $ there exists some  $ \K\prec_{s}\M $ containing $ \bar{a}$ such that $ (g_{0}\upharpoonright_{K},...,g_{n}\upharpoonright_{K})\in(\A_{\NN}^{\M}(\K))^{n+1}  $ is e.c.
		\item[(ii)] For every two $ \K,\mathcal{L}\prec_{s}\M $ and  for each $ (f_{0},...,f_{n})\in\A_{\NN}^{\M}(\mathcal{L}) $ if  the following conditions $(\spadesuit)$ hold, then there exists some $ h\in\G_{(K)} $ such that $ f_{k}\subseteq g_{k}^{h} $ for all $k=0,...n$.
		
	$$(\spadesuit): \ \ \left(\begin{array}{c}
	\K\prec\mathcal{L}, \text{ and }	\\
	{g_{k}\upharpoonright_{K}=f_{k}\upharpoonright_{K}\in\A_{\NN}^{\M}(\K) }  \text{ is e.c. for all } k=0,...,n
	\end{array}\right).$$
	\end{itemize}
Let $ \LG(\M) $ indicates to the class of all finite tuples of Lascar generic automorphisms of $ \M $. When there is no risk of confusion, we simply write $ \LG $ .
\end{dfn}

\begin{theorem}[Lascar \cite{las}]
$(n\in\omega)$	The set of all $(n+1)$-tuples of Lascar generic automorphisms of a countable arithmetically saturated model $ \M $ of $ \pa $ is comeagre in $ (\A(\M))^{n+1} $.
\end{theorem}
The following lemma, is essentially due to Lascar \cite{las} and can be proved by an argument similar to that used in the proof of  \cite[Propositon 6.2]{km}:
\begin{lem}
Suppose $\M$ is a countable $\mathbb{L}$-structure (where $\mathbb{L}$ is an arbitrary countable language) and $\G$ is the Polish group of its automorphisms. Moreover, let $X$ be a comeagre subset of $\G$, $\mathrm{H}$ be a nonopen subgroup of $\G$, and $\mathrm{O}$ be a nonempty open subset of $\G$. Then there exists some $g\in X\cap\mathrm{O}$ which is not inside $\mathrm{H}$.
\end{lem}
\begin{con}
Within a  model $\M$ of $\pa$, we can   canonically\footnote{By a canonical enumeration of $\mathbb{L}_{A}$-formulas or $\mathbb{L}_{A}$-terms, we mean an enumeration in which each index represents the G\"{o}del number of the corresponding $\mathbb{L}_{A}$-formula or $\mathbb{L}_{A}$-term.}  enumerate all $\mathbb{L}_{A}$-formulas (containing the nonstandard $\mathbb{L}_{A}$-formulas in $\M$) with a specific number of variables, as $\langle \bm{\phi}_{r}(\bar{x}): \ r\in M\rangle$, and all $\mathbb{L}_{A}$-terms (containing the nonstandard $\mathbb{L}_{A}$-terms in $\M$) with a specific number of variables, as $\langle \textbf{t}_{r}(\bar{x}): r\in M\rangle$.  Throughout this paper, we frequently refer to these canonical enumerations, and the number of variables used in the $\mathbb{L}_{A}$-formulas or $\mathbb{L}_{A}$-terms will be clear from  context.
\end{con}
\section{Small index property }
In this section, we begin by generalizing the argument originally used by Lascar in \cite{las}  to prove the SIP for the automorphism group of countable arithmetically saturated models of $\pa$. We then apply this generalized method to establish the small index property  for  the subgroup of the automorphism group of a countable recursively saturated model of $\pa$ that pointwise stabilizes a strong cut of the model.

\subsection{Lascar's method}
Here, we define a notion based on Lascar's proof of Theorem 2.5:

Suppose $\mathbb{L}$ is a countable language,  $\M$ is a countable $\mathbb{L}$-structure, and $\G$ is the  Polish group of $\A(\M)$. We say that the triple $(\mathcal{X},\mathcal{A},\{Y_{K}: \ K\in\mathcal{A}\})$ is a\textit{ Lascar generic system} for $\G$ if:
\begin{itemize}
	\item[(1)] $\mathcal{X}:=\bigcup_{n>0}\mathcal{X}_{n}$, where each $\mathcal{X}_{n} $ is a comeagre subset of $\G^{n}$ (with product topology).
	 \item[(2)]   ($n>0$) For every $(f_{1},...,f_{n})\in\mathcal{X} $ and for each  $h\in\G$, it holds that $(f_{1}^{h},...,f_{n}^{h})\in\mathcal{X} $.
	\item[(3)]  ($n>0$) For every $(f_{1},...,f_{n})\in \mathcal{X}$ the following set is a comeagre subset of $\G$:
	$$\{g\in \G: \ (f_{1},...,f_{n},g)\in \mathcal{X}_{n+1}\}. $$
	\item[(4)] $\mathcal{A}$ is a family of subsets of $M$ such that $\G_{(K)}$ is open in $\G$ for every $K\in\mathcal{A}$.
	\item[(5)] For every $K\in\mathcal{A}$ and for each $h\in \G$, $h(K)\in\mathcal{A}$.
	\item[(6)] For every $K \in \mathcal{A}$, let $Y_K := \bigcup_{n >0} Y_{K,n}$, where each $Y_{K,n}$ is a set of $n$-tuples such that each element of the $n$-tuple is the restriction to $K$ of an automorphism $g$ of $\M$ satisfying $g(K) = K$.
	\item[(7)]  ($n>0$) For every tuple $(f_{1},...,f_{n})\in\mathcal{X} $, for each  $\bar{a}\in M$ and for every $K\in\mathcal{A}$ there exists some $L\in \mathcal{A}$ containing $\bar{a}$ and $K$ such that $(f_{1}\upharpoonright_{L},...,f_{n}\upharpoonright_{L})\in Y_{K,n} $.
	 \item[(8)]  ($n>0$) For each tuple $(f_{1},...,f_{n})\in\mathcal{X} $ and for every  $h\in\G$, if $K\in \mathcal{A}$ such that ${(f_{1}\upharpoonright_{K},...,f_{n}\upharpoonright_{K})\in Y_{K,n}}$, then $(f_{1}^{h}\upharpoonright_{h^{-1}(K)},...,f_{n}^{h}\upharpoonright_{h^{-1}(K)})\in Y_{K,n}$.
	\item[(9)]  ($n>0$) For every two $n$-tuples $(f_{1},...,f_{n}),(g_{1},...,g_{n})\in \mathcal{X}_{n}$ if there exists some $K\in\mathcal{A}$ such that $(f_{1}\upharpoonright_{K},...,f_{n}\upharpoonright_{K})\in Y_{K,n}$ and $f_{i}\upharpoonright_{K}=g_{i}\upharpoonright_{K}$
 for all $i=1,...,n$, then there exists some $h\in\G_{(K)}$ such that $f_{i}=g_{i}^{h}$ for all  $i=1,...,n$.
\end{itemize}
 We will show that any automorphism group possessing a Lascar generic system has the SIP. For this purpose we need the following lemmas:

 \begin{lem}[Hodges-Hodkinson-Lascar-Shelah \cite{hod}]
Suppose $\G$ is a Polish group. Then the index of any meagre subgroup of $\G$ is $2^{\aleph_{0}}$.
 \end{lem}
 
 \begin{lem}
 $(n\in\omega)$	Suppose $\M$ is a countable $\mathbb{L}$-structure (for a given countable language $\mathbb{L}$) and $\G:=\A(\M)$. Moreover, let  $(\mathcal{X},\mathcal{A},\{Y_{\K}: \K\in\mathcal{A}\})$ be a Lascar generic system for $\G$ and $(g_{0},...,g_{n})\in\mathcal{X}$. Then:
 \begin{itemize}
\item[(1)] If   $\mathrm{H}$ is a subgroup of $\G$ whose index is less than $2^{\aleph_{0}}$, then there exists some $g\in \mathrm{H}$ such that $(g_{0},...,g_{n},g)\in\mathcal{X}$.
\item[(2)] If $\mathrm{H}$ is a nonopen subgroup of $\G$ and  $ \mathrm{O}$ is a nonempty open subset of $\G$, then there exists some $g\in \mathrm{O}\setminus \mathrm{H}$ such that $(g_{0},...,g_{n},g)\in\mathcal{X}$.
 \end{itemize}
 \end{lem}
 \begin{proof}
 	\begin{itemize}
 	\item[(1)]  Suppose not. So by the third condition in the definition of a Lascar generic system, $\mathrm{H}$ is meagre in $\G$; but this is in contradiction with the previous lemma.
 	\item[(2)] By the third condition in the definition of a Lascar generic system and Lemma 2.8. 
 	\end{itemize}
\end{proof}

We are now ready to outline a generalization of the steps used by Lascar in his proof:
\begin{theorem}
	Suppose $\M$ is a countable $\mathbb{L}$-structure (for a given countable language $\mathbb{L}$) and $\G:=\A(\M)$ such that it possesses a Lascar generic system $(\mathcal{X},\mathcal{A},\{Y_{\K}: \K\in\mathcal{A}\})$. Then $\G$ has the SIP.
\end{theorem}
\begin{proof}[Sketch of proof]
 Suppose $\mathrm{H}$	is a nonopen subgroup of $\G$ whose index is less than $2^{\aleph_{0}}$.  By using Lemmas 3.1 and 3.2, we will inductively build  sequence $(M_{s}: \ s\in 2^{<\omega})$ of elements of $\mathcal{A}$, and sequences $(g_{s}: \ s\in 2^{<\omega})$ and $(h_{s}: \ s\in 2^{<\omega})$ of elements of $\G$ such that:
\begin{itemize}
\item[(1)] For every $\sigma\in 2^{\omega}$ it holds that $M=\bigcup_{n>0}M_{\sigma\upharpoonright_{n}}$.
\item[(2)] For every $s,t\in 2^{<\omega}$ if $s\sqsubset_{e} t$ (i.e., $s$ is an initial segment of $t$) then $M_{s}\subseteq M_{t}$.
\item[(3)] For every  $s\in 2^{<\omega}$, the tuple $(g_{\emptyset},g_{s\upharpoonright_{1}},g_{s\upharpoonright_{2}},...,g_{s})$ is in $\mathcal{X} $.
\item[(4)] For every $s,t,t'\in 2^{<\omega}$ if $s\sqsubset_{e} t$  and $s\sqsubset_{e} t'$ then $h_{t}\upharpoonright_{M_{s}}=h_{t'}\upharpoonright_{M_{s}}$.
\item[(5)] For every $s,t,t'\in 2^{<\omega}$ if $s\sqsubset_{e} t\sqsubset_{e} t'$  then $g_{s}^{h_{t}}=g_{s}^{h_{t'}}$
\item[(6)] $g_{s}^{h_{s^{\smallfrown}0}}\in H$ and $g_{s}^{h_{s^{\smallfrown}1}}\notin H$, for every $s\in 2^{<\omega}$ (here, by $ s^{\smallfrown}0$ and $s^{\smallfrown}1$, we mean the sequences which is obtained from appending $0$ and $1$ to $s$, respectively).
\end{itemize}
Suppose we have built such sequences. As a result, for every $\sigma\in 2^{\omega}$ the sequence $\{h_{\sigma\upharpoonright_{n}}: \ n\in\omega\}$ is Cauchy. So, we put ${h_{\sigma}:=\lim\limits_{n\rightarrow\infty}h_{\sigma\upharpoonright_{n}}}$. From the above statements, we conclude that $h_{\sigma}H\neq h_{\tau}H$ for every distinct $\sigma,\tau\in 2^{\omega}$, which contradicts the assumption that $[\G:\mathrm{H}]<2^{\aleph_{0}}$.

In order to construct the  aforementioned sequences, suppose $\{a_{n}: \ n\in\omega\}$ is an enumeration of $M$.
For the first step of the inductive construction, since we have assumed that ${[\G:\mathrm{H}]<2^{\aleph_{0}}}$, by Lemma 3.2(1), there exists some $g_{\emptyset}\in\mathrm{H}\cap\mathcal{X}$. Then by the definition of a Lascar generic system, there exists some $M_{\emptyset}\in \mathcal{A}$ containing $a_{0}$ such that $g_{\emptyset}\upharpoonright_{M_{\emptyset}}\in Y_{M_{\emptyset}}$. Moreover, put $h_{\emptyset}=h_{0}=id$. Then, by Lemma 3.2(2) there exists some $f\in g_{\emptyset}\G_{(M_{\emptyset})}\setminus \mathrm{H}$ such that $f\in\mathcal{X}$. Again, by the definition of a Lascar generic system, there exists some $h_{1}\in\G_{(M_{\emptyset})}$ such that $g_{\emptyset}^{h_{1}}=f$.

For the induction step, suppose $s\in 2^{<\omega}$  is given. Let $n\in\omega$ be the length of $s$  and suppose $g_{s}$,  $M_{s}$, $h_{s}$, $h_{s^{\smallfrown}0}$, and $h_{s^{\smallfrown}1}$ has been built. Let $t:=s^{\smallfrown}0$ or $t:=s^{\smallfrown}1$. So $[\G:\mathrm{H}^{h_{t}}]<2^{\aleph_{0}}$. Then, by Lemma 3.2(1) there exists some $g_{t}\in \mathrm{H}^{h_{t}}$ such that ${(g_{s\upharpoonright_{0}},g_{s\upharpoonright_{1}},...,g_{s},g_{t})\in\mathcal{X}}$. Then, put ${h_{t^{\smallfrown}0}:=h_{t}}$. By the definition of a Lascar generic system $(g^{h_{t}}_{s\upharpoonright_{0}},g^{h_{t}}_{s\upharpoonright_{1}},...,g^{h_{t}}_{s},g_{t}^{h_{t}})\in\mathcal{X} $. As a result, there exists some $K\in\mathcal{A}$  such that ${(g^{h_{t}}_{s\upharpoonright_{0}}\upharpoonright_{K},g^{h_{t}}_{s\upharpoonright_{1}}\upharpoonright_{K},...,g^{h_{t}}_{s}\upharpoonright_{K},g_{t}^{h_{t}}\upharpoonright_{K})\in Y_{K}} $ and $K$ contains $h_{t}(a_{n+1})$ and $h_{t}(M_{s})$.  Now, let $M_{t}:=h_{t}^{-1}(K)$. Again, by Lemma 3.2(2) there exists some $f'\in g_{t}^{h_{t}}\G_{(M_{t})}\setminus\mathrm{H}$ such that $ (g^{h_{t}}_{s\upharpoonright_{0}},g^{h_{t}}_{s\upharpoonright_{1}},...,g^{h_{t}}_{s},f')\in\mathcal{X}$. Thus, by the definition of a Lascar generic system there exists some $h\in\G_{(M_{t})}$ such that $f'=g_{t}^{h_{t}h}$ and $g_{s\upharpoonright_{m}}^{h_{t}}=g_{s\upharpoonright_{m}}^{h_{t}h}$ for all $m\leq n$. Finally, put $h_{t^{\smallfrown}1}:=h_{t}h$.
\end{proof}

\subsection{The SIP for $(\A(\M))_{(I)}$}

Through this subsection, suppose $\M$ is a countable recursively saturated model of $\pa$, $I $ is a strong cut of $\M$, and  $\G:=\A(\M)$. This subsection is devoted to constructing a Lascar generic system for the Polish group $\G_{(I)}$ (as a closed subgroup of the Polish group $\G$). To formulate the problem in the framework introduced in the previous subsection, let  $\mathbb{L}$ be the language of $\mathbb{L}_{A}$ augmented with countably many   constant symbols naming elements of  $I$.  Then by considering the expansion $\M^{*}$ of $\M$ to the language $\mathbb{L}$, the automorphism group  $\A(\M^{*})$ is precisely $\G_{(I)}$. Our goal is to construct a Lascar generic system for this automorphism group. For this purpose, we will generalize the notions of an existentially closed  and a Lascar generic automorphism:

 \begin{dfn}
	 ($ n\in\omega$) Suppose   $ \K\prec\M $ is $I$-small, and $ (g_{0},...,g_{n})\in(\Aa)^{n+1} $. We say $ (g_{0},...,g_{n}) $ is \textit{$I$-existentially closed} (abbreviated by \textit{$I$-e.c.}) if for every tuple  $ {(f_{0},...,f_{n})\in(\G_{(I)})^{n+1}} $ such that $f_{k}\upharpoonright_{K}=g_{k}\upharpoonright_{K}$ for all $k=0,...,n$,  for every $\mathbb{L}_{A}$-formula $ \varphi(x,y_{0},...,y_{n},\bar{z}) $, and for each $ \bar{c}\in K $, if $ \M\models\exists x \ \varphi(x,f_{0}(x),...,f_{n}(x),\bar{c}) $, then $ {\M\models\varphi(d,g_{0}(d),...,g_{n}(d),\bar{c}) }$ for some $ d\in K $.
	
	Let $\mathrm{EC}_{I}(\K)$ be the family of all finite tuples of  $I$-e.c. elements of $\Aa$.
\end{dfn}

First, using an argument analogous to the one establishing the existence of existentially closed automorphisms when the standard cut is strong in $\M$, we will prove that if $I$ is a strong cut of $\M$ there exist plenty of $I$-existentially closed automorphisms in $\M$:
\begin{theorem}
	($n\in\omega$)	Suppose $ \M $ is a countable  recursively saturated model of $ \pa $, $I$ is a strong cut of $\M$, and $ a,b_{0},...,b_{n}\in M $ such that $ \tp(a,i)=\tp(b_{k},i) $ for every $k=0,...,n$ and for all $i\in I$. Then there exists some $I$-small elementary submodel $ \K $ of $\M$  and some $I$-e.c.  tuple $ (g_{0},...,g_{n})\in(\Aa)^{n+1} $ such that $ a,b_{0},...,b_{n}\in K $ and $ g_{k}(a)=b_{k} $ for every $k=0,...,n$.
\end{theorem}
\begin{proof}
	We will prove the theorem for the case $n=0$, and the remaining cases follow by a similar argument.  So for simplicity, let $b:=b_{0}$.\\
		Let $S$ be an inductive satisfaction class for $\M$ such that $(\M;S)$ is also recursively saturated. So by Lemma 2.2(1), there exists some  $I$-small elementary submodel  $(\K;S')$  of $(\M;S)$ containing $a,b$.  Let $\K=\{(\alpha)_{i}: \ i\in I\}$ for some $\alpha\in M$. Moreover, let $a=(\alpha)_{i_{a}}$ and $b=(\alpha)_{i_{b}}$ for some $i_{a},i_{b}\in I$. Now, we  define:
	$$\Delta(x,y):=\mathrm{max}\{z: \ \forall \langle w,i\rangle<z \ S(\bm{\phi}_{w}(x,i))\leftrightarrow S(\bm{\phi}_{w}(y,i))\}\footnote{Throughout this paper, we adopt the convention that $\max(\emptyset)=0$.};$$ 
	and
	$$\Lambda(\langle i,j\rangle):=\Delta((\alpha)_{i},(\alpha)_{j})\footnote{For case $n>0$, we define: $$\Lambda(\langle i_{0},i_{1},...,i_{n+1}\rangle):=\min\{\Delta((\alpha)_{i_{k}},(\alpha)_{i_{m}}): \ 0\leq k\leq m\leq n+1  \}. $$}. $$
	
	Since $I$ is strong in $\M$,	there exists some $\epsilon>I$ such that $\Lambda(i)\in I$ iff $\Lambda(i)<\epsilon $ for all $i\in I$.  Then, let	$\langle\varphi_{r}((\alpha)_{i},x,y): \  r,i\in M\rangle$ be a  canonical enumeration of $\mathbb{L}_{A}$-formulas of the given form with $(\alpha)_{i}$s as parameters.  Moreover, let $\Theta(i,u,v,x,y)$ be the following $\mathbb{L}$-formula (where $\mathbb{L}:=\mathbb{L}_{A}\cup\{S\}$):
	$$  \left(\begin{array}{c}
		S(\varphi_{(i)_{0}}((\alpha)_{(i)_{1}},(\alpha)_{x},(\alpha)_{y})) \ \wedge \\ \forall z,w  \left(\begin{array}{c}
	   	(\alpha)_{z}=(\alpha)_{u}^{\frown}(\alpha)_{x}\ \wedge  \
		(\alpha)_{w}=(\alpha)_{v}^{\frown}(\alpha)_{y}
	\end{array}\right)  
	\rightarrow\Lambda(\langle z,w\rangle)>\epsilon\end{array}\right). $$
	Now,  we inductively define  $\mathbb{L}$-terms $\Omega(i)$ and $\bar{\Omega}(i)$ in $\M$ as follows:
	\begin{center}
	$\Omega(0):=\langle i_{a},i_{b}\rangle$, and\\ $\bar{\Omega}(0):=\langle \textbf{u},\textbf{v}\rangle\in I$ such that $(\alpha)_{\textbf{u}}=\langle a\rangle$ and $(\alpha)_{\textbf{v}}=\langle b\rangle$.
	\end{center} 
	and
	\begin{equation*}
		\Omega(i+1):=  \left\{
		\begin{array}{rl}
			\mu_{\langle x,y\rangle}   \Theta( i,(\bar{\Omega}(i))_{0},(\bar{\Omega}(i))_{1},x,y) & \text{if } \scriptstyle(\M;S)\models\exists x,y \    \Theta( i,(\bar{\Omega}(i))_{0},(\bar{\Omega}(i))_{1},x,y)\\
			\langle i_{a},i_{b}\rangle \   \ \ \ \ \ \ \ \ \ \ \ \ \ \ \ \ \ \ \ \ \ \ \ \ \ \ \ \ \ \ \    & \text{o.w.} 
		\end{array} \right.
	\end{equation*}
\begin{equation*}
	\bar{\Omega}(i+1):= \left\{
	\begin{array}{rl}
	\mu_{\langle u,v \rangle}  \left(\begin{array}{c}
		(\alpha)_{u}=((\alpha)_{(\bar{\Omega}(i))_{0}})^{\frown}(\alpha)_{(\Omega(i+1))_{0}}  \ \wedge \\ (\alpha)_{v}=((\alpha)_{(\bar{\Omega}(i))_{1}})^{\frown}(\alpha)_{(\Omega(i+1))_{1}}
	\end{array}\right) & \text{if }\scriptsize\exists u,v 
 \left(\begin{array}{c}
	(\alpha)_{u}=((\alpha)_{(\bar{\Omega}(i))_{0}})^{\frown}(\alpha)_{(\Omega(i+1))_{0}}  \ \wedge \\ (\alpha)_{v}=((\alpha)_{(\bar{\Omega}(i))_{1}})^{\frown}(\alpha)_{(\Omega(i+1))_{1}}
\end{array}\right)\\
\epsilon \ \ \ \ \ \ \ \ \ \ \ \ \ \ \ \ \ \ \ \ \ \ \ \ \ \ \ \ \ \ \ \ \ \ \ \ \ \ \ \ \ \ \ \ \ \ \ \ \ \ \ \ \ \ \ \ \  & \text{o.w.}
\end{array} \right.
\end{equation*}	
	Again, since $I$ is strong in $\M$	there exists some $\eta,\varepsilon>I$ such that $\Omega(i)\in I$ iff ${\Omega(i)<\varepsilon} $, and  $\bar{\Omega}(i)\in I$ iff $\bar{\Omega}(i)<\eta $ for all $i\in I$. Note that $\Omega(i)<\varepsilon $ and $\bar{\Omega}(i)<\eta$ for all $i\in I$: we use induction to show this fact. Suppose we have shown $\Omega(i)<\varepsilon $ and $\bar{\Omega}(i)<\eta$ for $i\in I$. By Lemma 2.2(3), it suffices to show that $\Omega(i+1)<\varepsilon$. Suppose ${(\M;S)\models\exists x,y \    \Theta( i,(\bar{\Omega}(i))_{0},(\bar{\Omega}(i))_{1},x,y)}$ (for the other case, it is clear that $\Omega(i+1)<\varepsilon$). So, for every $s\in I$ it holds that:\\
	
$(1): \ \ \ $	$(\M;S)\models\exists x,y  \left(\begin{array}{c}
		S(\varphi_{(i)_{0}}((\alpha)_{(i)_{1}},x,y)) \ \wedge \\ 
		\forall u,v  \left(
		\begin{array}{c}
			u=((\alpha)_{(\bar{\Omega}(i))_{0}})^{\frown}x \ \wedge \\ v=((\alpha)_{(\bar{\Omega}(i))_{1}})^{\frown}y	\end{array}\right)\rightarrow
			\forall \langle r,z\rangle<s \
		 S(\bm{\phi}_{r}(u,z))\leftrightarrow S(\bm{\phi}_{r}(v,z)) 
		\end{array}\right)$.\\
	
	So since $(\K;S')\prec(\M;S)$, statement (1) implies that for every $s\in I$ we have:\\
	
	$(2): \ \ \ $	$(\K;S')\models\exists x,y  \left(\begin{array}{c}
		S'(\varphi_{(i)_{0}}((\alpha)_{(i)_{1}},x,y)) \ \wedge \\ 
		\forall u,v \left(
		\begin{array}{c}
			u=((\alpha)_{(\bar{\Omega}(i))_{0}})^{\frown}x \ \wedge \\ v=((\alpha)_{(\bar{\Omega}(i))_{1}})^{\frown}y	\end{array}\right)\rightarrow
		\forall \langle r,z\rangle<s \
		S'(\bm{\phi}_{r}(u,z))\leftrightarrow S'(\bm{\phi}_{r}(v,z)) 
	\end{array}\right)$.\\
	
As a result, by using  Overspill principle over $ I$ in $(\K;S')$ and Lemma 2.2(3), we infer that $\Omega(i+1)\in I$.

		Now, we define the following recursive type for every $s\in M$:
	\begin{align*}
		p_{s}(x,y):=&\{\forall i<s \ \phi(x,i)\leftrightarrow\phi(y,i): \ \phi \text{ is an } \mathbb{L}_{A}\text{-formula}\}\cup\\&
	\left\lbrace	
		\forall i<s \ 
	 ((x)_{i}=(\alpha)_{(\Omega(i))_{0}} \ \wedge\ (y)_{i}=(\alpha)_{(\Omega(i))_{1}}) 
	\right\rbrace.
\end{align*}

	First, we will prove that there exists some  $s>I$ such that  $p_{s}(x,y)$ is finitely satisfiable. For this purpose, define:
	\begin{center}
		$\Psi(m):= \max\left\lbrace s: \exists x,y \ 	\psi(x,y,s,m,\alpha) \right\rbrace$, where $\psi(x,y,s,m,\alpha) $ is the following $\mathbb{L}$-formula:
		$$\forall i<s \ \forall z<m \left(\begin{array}{c}
			(S(\bm{\phi}_{z}(x,i))\leftrightarrow S(\bm{\phi}_{z}(y,i))) \ \wedge\\
				 ((x)_{i}=(\alpha)_{(\Omega(i))_{0}} \ \wedge\ (y)_{i}=(\alpha)_{(\Omega(i))_{1}}) 
		\end{array}\right)  .$$
	\end{center}
	
	Let $e>I$ be the witness of strength of $I$ in $\M$ for the definable function $\Psi$. We will  show that $p_{e}(x,y)$ is finitely satisfiable. For this purpose, it suffices to prove that $\Psi(m)>e$ for every $m\in\omega$. So, suppose $m\in\omega$ be arbitrary. By the way we defined $\Omega(i)$ and by using induction in $(\M;S)$, we can show  that ${(\M;S)\models\exists x,y \ \psi(x,y,s,m,\alpha)}$ for every $s\in I$. So, by Overspill principle over $I$ in $(\M;S)$, we have $\Psi(m)>e$.
	
	As a result,  we can find some $ \lambda,\xi\in M $ such that $(\lambda,\xi)$ realizes $p_{e}(x,y)$. Now, put:
	$$ g:=\bigcup_{i\in I}(\lambda)_{i}\mapsto(\xi)_{i} .$$ 
	Then, we have the following claim:
	\begin{center}
		\textit{Claim:} For every $m\in\omega$ and for each $i\in I$, if  $\M\models\exists x \ \varphi_{m}((\alpha)_{i},x,f(x))$ for some $f\in\G_{I}$ such that $g\subseteq f$, then $\M\models\varphi_{m}((\alpha)_{i},(\lambda)_{j},(\xi)_{j})$ for some $j\in I$.
	\end{center}
	
The claim implies that, first $\mathrm{Dom}(g)=K$:  note that since $\Omega(i)\in I$ for all $i\in I$, it holds that $\mathrm{Dom}(g)\subseteq K $. To see $K\subseteq \mathrm{Dom}(g)$, let $i\in I$ be arbitrary and consider  the $\mathbb{L}_{A}$-formula $(x=(\alpha)_{i} \ \wedge \ y=y)$ and $f\in\G_{(I)} $ such that $f(\lambda)=\xi$ in the claim. So we conclude that $(\alpha)_{i}\in \mathrm{Dom}(g)$. In a similar way, $\mathrm{Range}(g)=K$. Secondly, the claim implies that $g\in\Aa$ is $I$-e.c.
	
	In order to prove the claim, suppose we have the assumptions of the Claim. Thus, by iterating similar arguments to statements (1) and (2), we conclude that $\M\models\varphi_{m}((\alpha)_{i},(\lambda)_{j},(\xi)_{j})$ where $j:=\Omega(\langle m,i\rangle+1)$.
\end{proof}
Now, we will adapt the notion of a Lascar generic automorphism:	
\begin{dfn}
($n\in\omega$)	A tuple  $ (g_{0},...,g_{n})\in(\G_{(I)})^{n+1} $ is called \textit{Lascar $I$-generic} if:
	\begin{itemize}
		\item[(i)] For every finite tuple $ \bar{a}\in M $ there exists some $I$-small elementary submodel $ \K $ of $\M$ such that $ \bar{a}\in K $ and $ (g_{0}\upharpoonright_{K},...,g_{n}\upharpoonright_{K})\in(\Aa)^{n+1}  $ is $I$-e.c.
		\item[(ii)] For every two $I$-small elementary submodel $ \K$ and $ \mathcal{L}$ of $\M$, and for each tuple $ (f_{0},...,f_{n})\in(\A_{(I)}^{\M}(\mathcal{L}))^{n+1} $, if  the following conditions $(\spadesuit)$ hold, then there exists some $ h\in\G_{(K)} $ such that  $ f_{k}=g_{k}^{h}\upharpoonright_{L} $ for all $k=0,...,n$.
		$$(\spadesuit): \ \ \left(\begin{array}{c}
		 \K\prec\mathcal{L},\\
		 {g_{k}\upharpoonright_{K}=f_{k}\upharpoonright_{K}\in\Aa }   \text{ for all } k=0,...,n,  \text{ and }	\\
		(g_{0}\upharpoonright_{K},...,g_{n}\upharpoonright_{K})  \text{ is } I\text{-e.c.}
		\end{array}\right).$$
	\end{itemize}
Let $ \LG_{I}(\M) $ indicates to the class of all finite tuples of Lascar generic automorphisms in $ \G_{I} $. When there is no risk of confusion, we simply write $ \LG_{I} $ .
\end{dfn}

Clearly, $\LG_{\NN}=\LG$. Moreover,
it is easy to verify that $\LG_{I}$  is closed under conjugacy in $\G_{(I)}$. As in the case of Lascar generic automorphisms, we have the following criterion for two Lascar $I$-generic automorphism to be conjugates:
\begin{theorem}
($n\in\omega$) 	Suppose $ \M $ is a countable  recursively saturated model of $ \pa $ and $I$ is a strong cut of $\M$.  Moreover, let  $(g_{0},...,g_{n})$ and $(f_{0},...,f_{n})$ be two tuples of Lascar $I$-generic automorphisms of $\M$,  $\K$ be an $I$-small elementary submodel of $\M$ such that ${g_{k}\upharpoonright_{K}=f_{k}\upharpoonright_{K}\in\Aa}$ for all $k=0,...,n$, and $(g_{0}\upharpoonright_{K},...,g_{n}\upharpoonright_{K})$ be $I$-e.c. Then there exists some $h\in\G_{(K)}$ such that $g_{k}=f_{k}^{h}$ for all $k=0,...,n$.
\end{theorem}
\begin{proof}
The proof is identical to that given in \cite{ksb} for the analogous result concerning Lascar generic automorphisms.	We will prove the theorem for the case $n=0$, and the remaining cases follow by a similar argument. For simplicity put $g:=g_{0}$ and $f:=f_{0}$. Let $ \{a_{n}: n\in\omega\} $ be an enumeration of elements of $ M $. Since $ g $ is Lascar $ I $-generic, there exists some $ I $-small elementary submodel  $ \K_{0} $ of $\M$ containing $ K $ and $ a_{0} $ such that $ g\upharpoonright_{K_{0}}\in \A^{\M}(\K_{0}) $ is $ I $-e.c. So by the definition of an Lascar $ I $-generic automorphism, there exists some $ h_{0}\in\G_{(K)} $ such that $ g\upharpoonright_{K_{0}}\subseteq f^{h_{0}} $. Again, since $ f^{h_{0}} $ is Lascar $ I $-generic, there exists some $ I $-small $ \K_{1}\prec\M $ containing $ K_{0} $ and $ a_{1} $ such that $ f^{h_{0}}\upharpoonright_{K_{1}}\in \A^{\M}(\K_{1}) $ is $ I $-e.c. Then, there exists some $ h_{1}\in\G_{(K_{0})} $ such that $ f^{h_{0}}\upharpoonright_{K_{1}}\subseteq g^{h_{1}} $. By iterating this argument, we will find a sequence  $ \{h_{m}: \ m\in\omega\} $ of automorphisms of $ \M $ such that the sequences $ \{h_{1}h_{3}...h_{2m+1}: \ m\in\omega\} $ and $ \{h_{0}h_{2}...h_{2m}: \ m\in\omega\} $ are Cauchy. As a result, we  put $ j:= \lim\limits_{m\rightarrow\infty} h_{1}h_{3}...h_{2m+1} $, $ l:= \lim\limits_{m\rightarrow\infty} h_{0}h_{2}...h_{2m}$, and $h:=lj^{-1}$. It is easy to see that $ f=g^{h} $.
\end{proof}
In the  next two theorems, we will prove that the family of Lascar $I$-generic automorphisms satisfies the first and the third conditions from the definition of a Lascar generic system (again, the argument and notation used in the proof of the following theorems are similar to those in \cite{las} and \cite{ksb}):
\begin{theorem}
($n\in\omega$) 	Suppose $ \M $ is a countable  recursively saturated model of $ \pa $ and $I$ is a strong cut of $\M$. Then the set of all $(n+1)$-tuples  of Lascar $I$-generic  automorphisms of  $ \M $  is comeagre in $ (\G_{(I)} )^{n+1}$.
\end{theorem}
\begin{proof}
	We will prove the theorem for the case $n=0$, and the remaining cases follow by a similar argument. For every $\bar{a}\in M$ define:
$$ D(\bar{a}):=\left\lbrace g\in\G_{(I)}: \begin{array}{c}
 \text{There exists some } I\text{-small }  \K\prec\M \text{ containing  } \bar{a} \text{ s.t.} \\ g\upharpoonright_{K}\in\Aa \text{ is } I\text{-e.c.}\end{array}\right\rbrace. $$
It is easy to see that $D(\bar{a})$ is open in $\G_{(I)}$. In addition,  by Theorem 3.4 it is  dense in $\G_{(I)}$.

Moreover,  for every two  $I$-small elementary submodel $ \K$ and $ \mathcal{L} $ of $\M$ and for each $ {f\in\A_{(I)}^{\M}(\mathcal{L})}$, define the  property $\blacklozenge(\K,\LL,f)$ as follows:
	$$ \left(\begin{array}{c}
	\K\prec\mathcal{L}, \text{ and }\\
f\upharpoonright_{K}\in\Aa \text{ is } I\text{-e.c.}
\end{array}\right).$$
  Now, if $\blacklozenge(\K,\LL,f)$ holds in $\M$,  let $ O(\K,\mathcal{L},f):=O_{1}(\K,f)\cup O_{2}(\K,\mathcal{L},f)$, where:
\begin{center}
	$O_{1}(\K,f):=\{g\in\G_{(I)}: \ g\upharpoonright_{K}\neq f\upharpoonright_{K}\},$ and \\
	$O_{2}(\K,\mathcal{L},f):=\{g\in\G_{(I)}: \exists h\in\G_{(K)} \ f= g^{h}\upharpoonright_{L} \}.$
\end{center} 
First, note  that both $O_{1}(\K,f)$ and $O_{2}(\K,\LL,f)$ are open in $\G_{(I)}$. Moreover, $O(\K,\mathcal{L},f)$ is dense in $\G_{(I)}$: to see this, let $\K=\{(\alpha)_{i}: \ i\in I\}$ and $\mathcal{L}=\{(\beta)_{i}: \ i\in I\}$ for some $\alpha,\beta\in M$, and  $\beta'=\hat{f}(\beta)$ for some extension $\hat{f}\in\G_{(I)}$ of $f$. Suppose $a,b\in M$ such that:
\begin{center}
$(1): \  \ \ $$\tp(a,\textbf{x},i)=\tp(b,f(\textbf{x}),i)$ for all $\textbf{x}\in K$ and every $i\in I$.
\end{center}    We shall find some $g\in\G_{(I)}$ and some $h\in\G_{(K)}$ such that $g^{h}(\beta)=\beta'$ and $g(a)=b$. For this purpose, for every $s\in M $ define:
\begin{align*}
		p_{s}(x,y):=&\{\forall i<s \ (\phi(x,a,i)\leftrightarrow\phi(y,b,i)): \ \phi  \text{ is an } \mathbb{L}_{A}\text{-formula}\}\cup\\ &
	\{  \forall i<s \ (\phi(\beta,\beta',(\alpha)_{i})\leftrightarrow\phi(x,y,(\alpha)_{i})): \ \phi  \text{ is an } \mathbb{L}_{A}\text{-formula}\}.
\end{align*}
We need to find some $s>I$ such that the recursive type $p_{s}(x,y )$ is finitely satisfiable. Let $S$ be an inductive satisfaction class for $\M$, and then  define: 
$$\Upsilon(m):= \max\left\lbrace s: \exists x,y \left(\begin{array}{c}
	\forall i<s \ \forall z<m \  (S(\bm{\phi}_{z}(x,a,i))\leftrightarrow S(\bm{\phi}_{z}(y,b,i))) \ \wedge\\
	\forall i<s  \ \forall z<m \ (S(\bm{\phi}_{z}(\beta,\beta',(\alpha)_{i}))\leftrightarrow S(\bm{\phi}_{z}(x,y,(\alpha)_{i})))
\end{array}\right)  \right\rbrace.$$
Since $I$ is a strong cut in $\M$, there exists some  $e\in M$ such that $\Upsilon(i)<I$ iff $\Upsilon(i)<e$ for all $i\in I$. We show that $p_{e}(x,y)$ is finitely satisfiable. It suffices to prove that $\Upsilon(m)>e$ for every $m\in\omega$. Otherwise, there exists some $m_{0}\in\omega$ such that $\Upsilon(m_{0})\leq e$. So, $\Upsilon(m_{0})\in I$. Put $s_{0}:=\Upsilon(m_{0})+1$. As a result, it holds that:\\

$(2): \ \ $$\M\models\forall x,y \left(\begin{array}{c}
	\forall i<s_{0} \ \bigwedge_{z=0}^{m_{0}} \  (S(\bm{\phi}_{z}(x,a,i))\leftrightarrow S(\bm{\phi}_{z}(y,b,i))) \rightarrow\\
	\exists i<s_{0}  \ \bigvee_{z=0}^{m_{0}} \ (S(\bm{\phi}_{z}(\beta,\beta',(\alpha)_{i})) \ \wedge \neg S(\bm{\phi}_{z}(x,y,(\alpha)_{i})))
\end{array}\right). $\\

Let:
$$ A:=\{\langle r,i\rangle<2^{\max\{s_{0},m_{0}\}}: (\M;S)\models r\leq m_{0} \ \wedge \ i<s_{0} \ \wedge \ S(\bm{\phi}_{r}(\beta,\beta',(\alpha)_{i}))\}.$$
So, there exists some $\epsilon\in I$ coding $A$ in $\M$. Then, put:
$$B:=\{\langle r,(\alpha)_{i}\rangle\in K: \ \M\models \langle r,i\rangle\E \epsilon \}. $$

Therefore,  by Lemma 2.2(3) there exists some $i_{0}\in I$ such that $(\alpha)_{i_{0}}$ codes $B$ in $\K$. Then, it holds that: \\

$(3): \ \ \ $$\M\models \exists x \  \forall z \ \bigwedge_{r=0}^{m_{0}} \ (\bm{\phi}_{r}(x,\hat{f}(x),z)\leftrightarrow \langle r,z\rangle\E(\alpha)_{i_{0}}) $.\\

So, since $f\upharpoonright_{K} $ is $I$-e.c., statement (3) implies that there exists some $d\in K$ such that:\\

$(4): \ \ \ $$\M\models    \forall z \ \bigwedge_{r=0}^{m_{0}} \ (\bm{\phi}_{r}(d,f(d),z)\leftrightarrow \langle r,z\rangle\E(\alpha)_{i_{0}}) $.\\ 

But statement (4) is in contradiction by statements (1) and (2). Therefore, $p_{e}(x,y)$ is finitely satisfiable. Consequently, by applying the back-and-forth method, we can construct the desired automorphisms  $g$ and $h$.

Moreover, we have: $$\LG_{I}=\bigcap_{\bar{a}\in M}D(\bar{a}) \ \cap\bigcap_{\substack{\K,\LL\in \mathcal{S}_{I},\\ f\in\A_{(I)}^{\M}(\LL),\\\M\models\blacklozenge(\K,\LL,f)}}O(\K,\LL,f).$$ 
As a result, $\LG_{I}$ is comeagre in $\G_{(I)}$.
\end{proof}

\begin{theorem}
	($n\in\omega$) 	Suppose $ \M $ is a countable  recursively saturated model of $ \pa $, $I$ is a strong cut of $\M$, and $(f_{0},...,f_{n})$ is an  Lascar $I$-generic tuple of automorphisms of $\M$. Then the following set is a comeagre subset of $\G_{(I)}$:
	$$X:=\{g\in \G_{(I)}: \ (f_{0},...,f_{n},g) \text{ is } I\text{-Lascar generic}\}. $$
\end{theorem}
\begin{proof}
	First, note that similar to the proof of the previous theorem, $X$ is equal to the intersection of countably many open sets. So it suffices to show that $X$ is dense in $\G_{(I)}$. For this purpose,  suppose $a\in M$ and $f\in\G_{(I)}$ are given. We shall find some $g\in X\cap f\G_{(I\cup\{a\})}$. Let $b:=f(a)$ and  $\K=\{(\alpha)_{i}: \ i\in I\}$ be some $I$-small elementary submodel of $\M$ containing $a,b$ such that $(f_{0}\upharpoonright_{K},...,f_{n}\upharpoonright_{K})$ is $I$-e.c. Now, by considering the open set $f_{0}(\G_{(I)})_{\alpha}\times...\times f_{n}(\G_{(I)})_{\alpha}\times f(\G_{(I)})_{a}$,  since by the previous theorem $(\LG_{I})^{n+2}$ is dense in $(\G_{(I)})^{n+2}$, there exists some Lascar $I$-generic tuple $(g_{0},...,g_{n},g_{n+1})$ such that $g_{k}\upharpoonright_{K}=f_{k}\upharpoonright_{K}$ for all $k=0,...,n$ and $g_{n+1}(a)=b$. Therefore, by Theorem 3.5, there exists some $h\in\G_{(K)}$ such that $f_{k}=g_{k}^{h}$ for all $k=0,...,n$. Put $g:=g_{n+1}^{h}$.
\end{proof}
Up to this point, we have verified all the conditions required by the definition of a Lascar generic system, leading to the following conclusion:
\begin{cor}
	Suppose $ \M $ is a countable  recursively saturated model of $ \pa $ and $I$ is a strong cut of $\M$. Then $(\LG_{I},\mathcal{S}_{I},\{\mathrm{EC}_{I}(\K): \ \K\in \mathcal{S}_{I}\})$ is a Lascar generic system for $\G_{(I)}$.
\end{cor}
Therefore, combining Theorem 3.3 and Corollary 3.8, we conclude that Theorem A stated in the Introduction section is established. 
\begin{q}
Can the assumption that 
$I$ is strong in $\M$ be removed from Theorem A?
\end{q}

\section{Some remarks concerning the  normal subgroups of $\A(\M)$}
In this section, we present some observations concerning Lascar $I$-generic automorphisms and their implications for the normal subgroups of $\A(\M)$.

We begin with the following lemma concerning the fixed points of every Lascar $I$-generic automorphism:

\begin{lem}
	Suppose $\M$ is a countable recursively saturated model of $\pa$ and $I$ is a  strong cut of $\M$. Then for every Lascar $I$-generic automorphism $g$ of $\M$ it holds that $ 	\I(g)=I $ and $ \fix(g)\setminus I$ is downward cofinal in $\M\setminus I $.
\end{lem}
\begin{proof}
	Suppose  $a\in M\setminus I$ is given. Since $ g $ is Lascar $I$-generic, there exists some $I$-small elementary submodel $ \K$ of $\M $ containing $ a $ such that $ g\upharpoonright_{K}\in\Aa $ is $I$-e.c.  Let $K=\{(\alpha)_{i}: \  i\in I\}$ for some $\alpha\in M$.
	
Note that	since $I$ is strong in $\M$,  the largest  common cut shared with $\M$ and ${\KK(\M;I\cup\{g(\alpha)\})}$ is $I$. As a result, there exists some $b_{1}>I$  such that $b_{1}<a$ and ${b_{1}\notin\KK(\M;I\cup\{g(\alpha)\})}$. So we can find some  $ f_{1}\in \G_{(I)} $ extending $ g\upharpoonright_{K} $ such that $ f(b_{1})\neq b_{1} $. As  a result,  since $g\upharpoonright_{K}$ is $I$-e.c. there exists some $c_{1}<a$ such that $g(c_{1})\neq c_{1}$. Therefore, $ 	\I(g)=I $.
	
	For the second statement of the theorem, let $S$ be a an inductive satisfaction class for $\M$ and define:
	$$\Phi(x):=\max\{y: \ \forall z,z'\leq y  \ \forall r\leq x \ S(\bm{\phi}_{r}(\alpha,z,z'))\leftrightarrow S(\bm{\phi}_{r}(g(\alpha),z,z'))\}. $$
	Let $e\in M$ be the witness of strength of $I$ for $\Phi$ in $\M$. Since for every $n\in\NN$ it holds that $\M\models \Phi(n)>e$, by using Overspill principle over $\NN$, there exists some $b_{2}>I$ such that $b_{2}<a$ and $\tp(\alpha,i,b_{2})=\tp(g(\alpha),i,b_{2})$ for all $i\in I$. As a result, there exists some $ f_{2}\in\G_{(I)} $ extending $g\upharpoonright_{K}  $ such that $ f_{2}(b_{2})=b_{2} $. Again, since $g$ is  $I$-e.c., we conclude that  $ \fix(g)\setminus I$ is downward cofinal in $\M\setminus I $.
\end{proof}
\begin{cor}
	
	Suppose $\M$ is a countable recursively saturated model of $\pa$, $\G:=\A(\M)$, and $I$ is a  strong cut of $\M$. Then there exists some automorphism $ f\in\G $ which is not  Lascar $I$-generic and $\I(f)=I$.
\end{cor}
\begin{proof}
	By Theorem 2.3 and Lemma 2.2(1)  there exists some $ f\in \G $ such that $ \fix(f)=\KK(\M;I) $. So $\I(f)=I$ (by strength of $I$). Since $I$ is strong in $\M$, there exists some $e>I$ such that there is no element of $\KK(\M;I) $ between $I$ and $e$. As a result, by Lemma 4.1,   $f$ is not Lascar $I$-generic.
\end{proof}
By Theorem 2.4(1), for every two cuts $I \subseteq J$ of $\M$ that are closed under exponentiation, the subgroup $\G_{(J)}$ is nowhere dense, and hence it is meagre in $\G_{(I)}$. Consequently, $\G_{(>I)}$, being a countable union of the subgroups $\G_{(J)}$ such that $I \subsetneqq J$, is also meagre in $\G_{(I)}$. However, by Lemma 4.1 we can conclude more than this:  
\begin{cor}
	Suppose $\M$ is a countable recursively saturated model of $\pa$, $\G:=\A(\M)$, and $I$ is a  strong cut of $\M$. Then $\G_{(>I)}\cap\LG_{I}=\emptyset$. 
\end{cor}
\begin{proof}
	A direct implication of Lemma 4.1.  
\end{proof}

As a result, the following holds for every strong cut $I$ of $\M$:
$$ \G_{(I)}=\bigcup_{g\in\LG_{I}}[g]^{\G_{(I)}} \ \bigcupdot \overset{\text{meagre in } \G_{(I)}}{\overbracket{  \bigcup_{\substack{g\notin\LG_{I}, \\ \I(g)=I}} [g]^{\G_{(I)}} \  \bigcupdot \ \bigcup_{g\in\G_{(>I)}}[g]^{\G_{(I)}}}}. $$  
In the remainder of this section, we show that Corollary 4.3 extends to all normal subgroups of $(\A(\M))_{(I)}$; i.e., we have the following theorem:
\begin{theorem}[or \textbf{Theorem B}]
	Suppose $\M$ is a countable recursively saturated model of $\pa$, $\G:=\A(\M)$, $I$ is a strong  cut of $\M$, and $\mathrm{N}$ is a normal subgroup of $\G_{I}$. 
	Then $\G_{(I)}=\mathrm{N}$ or $\mathrm{N}\cap\LG_{I}=\emptyset $. In particular, every nontrivial normal subgroup of $\G_{(I)}$ is meagre in it.
\end{theorem}

In order to prove Theorem 4.4, we first need to show that Theorem 2.4 (i.e., the Galois correspondence for closed normal subgroup of $\A(\M)$) can be extended to $(\A(\M))_{(I)}$ for every strong  cut $I$ of $\M$. For this purpose we use an argument nearly identical to that in \cite{sch1}:
\begin{theorem}
	Suppose $\M$ is a countable recursively saturated model of $\pa$, $\G:=\A(\M)$, and $I\subseteq J$ are two cuts of $\M$ closed under exponentiation.  Then: 
	\begin{itemize}
	\item[(1)]$J$ is invariant in $\G_{(I)}$ (i.e., $g(J)=J$ for every $g\in\G_{(I)}$) iff $J\cap\KK(\M;I)$ is cofinal in $J$ or $(M\setminus J)\cap\KK(\M;I)$ is downward cofinal in $M\setminus\KK(\M;I)$.
	\item[(2)] If $J$ is   invariant  in $\G_{(I)}$, then $\G_{(J)}$ is a closed normal subgroup of $\G_{(I)}$.
	\item[(3)] If $J=\inf \{\log^{n}(a): \ n\in \omega\} $ for some $a>I$, then  $\G_{(>J)}$ is equal to $\G_{(J')} $, where $ J':=\sup \{2_{n}^{a}: \ n\in \omega\}$. If $J$ is not of the above form, then the closure of $\G_{(>J)}$ in $\G_{(I)}$ equals to $\G_{(J)} $.
	\item[(4)]  For every closed normal subgroup $\mathrm{N}$ of $\G_{(I)}$, $\I(\mathrm{N})$ is closed under exponentiation and it is invariant in $\G_{(I)}$. Moreover, if $I$ is strong in $\M$, then  $\mathrm{N}=\G_{(\I(\mathrm{N}))}$.
	\end{itemize}
\end{theorem} 
\begin{proof}[Sketch of proof]
Parts (2) and (3) of the theorem follow directly from Theorem 2.4. For part (1), suppose $J$ is invariant in $\G_{(I)}$ and there are $a,b\in M$ such that $J\cap\KK(\M;I)<a\in J<b<(M\setminus J)\cap\KK(\M;I)$. Moreover, let $S$ be an inductive satisfaction class for $\M$. Then, since $a\notin\KK(\M;I)$ and $I$ is closed under exponentiation, for every $s\in I$ it holds that:\\

$(\M;S)\models\exists x \ (x>b \ \wedge \forall \langle r,i\rangle<s \ (S(\bm{\phi}_{r}(a,i))\leftrightarrow S(\bm{\phi}_{r}(x,i)))).$\\

As a result, by using Overspill principle over $I$ in $(\M;S)$, there exists some $f\in\G_{(I)}$ such that $f(a)>J$, which contradicts the assumption that $J$ is invariant in $\G_{(I)}$.

In order to prove part (4), we need the following lemmas (part (4) of the theorem is a direct result of Lemma 4.5.2 below):

\newtheoremstyle{nonum}{}{}{\itshape}{}{\bfseries}{.}{ }{#1 }
\theoremstyle{nonum}
\newtheorem{lemm}{Lemma 4.5.1}
\newtheoremstyle{nonum}{}{}{\itshape}{}{\bfseries}{.}{ }{#1 }
\theoremstyle{nonum}
\begin{lemm}
	Suppose $g\in\G_{(I)}$, $I'=\I(g)$, and there are arbitrary small element $\textbf{x}>I'$ such that $g(\textbf{x})<\textbf{x}$. Moreover, let $h\in\G_{(I')}$ such that $h(a)=b>a$ for some $a>I'$. Then there exist $\textbf{u},\textbf{v},\textbf{w}\in M$ such that:
	\begin{center}
		$ \textbf{u}<\textbf{v}<\textbf{w}$, $g(\textbf{v})=\textbf{u}$, and\\ $\tp(\textbf{u},\textbf{v},i)=\tp(\textbf{u},\textbf{w},i)$ and $\tp(\textbf{v},\textbf{w},i)=\tp(a,b
		,i)$ for all $i\in I$.
	\end{center}
\end{lemm}
\begin{proof}[Sketch of proof]
The proof of Lemma 4.5.1 is analogous to the argument in \cite{sch1} with some modifications. Below, we outline the proof sketch and highlight the parts that require adjustment:

First, we define:
$$ A:=\{\langle r,i\rangle\in I: \ (\M;S)\models S(\bm{\phi}_{r}(a,i))\},$$
$$ B:=\{\langle r,i\rangle\in I: \ (\M;S)\models S(\bm{\phi}_{r}(a,b,i))\}.$$
For every $s\in I$, there exist $\alpha_{s},\beta_{s}\in I$ which code $A\cap \{\textbf{x}\in M: \ \textbf{x}<s\}$ and ${B\cap \{\textbf{x}\in M: \ \textbf{x}<s\}}$, respectively (note that, both functions $s\mapsto\alpha_{s}$ and $s\mapsto\beta_{s}$ are coded in $\M$). 

Then,  let $\psi(r,s,u,v)$ be the following $\mathbb{L}$-formula (where $\mathbb{L}:=\mathbb{L}_{A}\cup\{S\}$):
$$\exists w \overset{\varphi(r,s,u,v,w)}{\overbrace{\left(\begin{array}{c}
	u<v<w\ \wedge \\
\forall i<s \ \forall k\leq r \ (
\langle k,i\rangle\E\alpha_{s}\leftrightarrow S(\bm{\phi}_{k}(u,i))) \ \wedge \\
\forall i<s \ \forall k\leq r \ (
\langle k,i\rangle\E\beta_{s}\leftrightarrow S(\bm{\phi}_{k}(v,w,i))) \ \wedge \\
\forall i<s \ \forall k\leq r \ (S(\bm{\phi}_{k}(u,v,i))
\leftrightarrow S(\bm{\phi}_{k}(u,w,i)))
\end{array}\right)}}. $$

Now, our plan is to define  $\mathbb{L}_{A}$-terms $l_{r}(s,x)$ and   definable $\M$-finite subsets $E_{r,s}$  for all $r,s\in I$ such that, intuitively, the $x$-th element of $E_{r,s}$ is $l_{r}(s,x)$, and $l_{r}(s,x)$ is the least element $v$ of $\M$ such that: (1)  $(v,i)$ realizes 'almost' the same type as the type of $(a,i)$ in $\M$ for all $i<s$, (2) for every $y<x$ if $u$  is the $y$-th element of   $E_{r,s}$ then $\psi(r,s,u,v)$ holds, and (3) for every $i<s$, for each $r'< r$, and for every $y\leq x$ if $ u$ is the $y$-th element of  $E_{r',i}$ then  $\psi(m,i,u,v)$ holds. To be more precise, we inductively define  $l_{n}(s,x)$  as follows:
\begin{center}
$l_{0}(0,0):=\mu_{z} \ (\langle 0,0\rangle\E\alpha_{0}\leftrightarrow \bm{\phi}_{0}(z,0)) ,$ 

$l_{0}(0,x+1):=\mu_{z} \ (\langle 0,0\rangle\E\alpha_{0}\leftrightarrow \bm{\phi}_{0}(z,0) \ \wedge \ 
\forall y\leq x \ \psi(0,0,l_{0}(0,y),z))
;$ 

$l_{r+1}(s+1,0):=\mu_{z} \left(\begin{array}{c}
	\forall i\leq s+1 \ \forall k\leq r+1 \ (
	\langle k,i\rangle\E\alpha_{s+1}\leftrightarrow \bm{\phi}_{k}(z,i)) \ \wedge \\
	\forall i<s+1 \ \forall k\leq r \ \psi(k,i,l_{k}(i,0),z) 
\end{array}\right),$\\

$l_{r+1}(s+1,x+1):=\mu_{z} \left(\begin{array}{c}
	\forall i\leq s+1 \ \forall k\leq r+1 \ (
	\langle k,i\rangle\E\alpha_{s+1}\leftrightarrow \bm{\phi}_{k}(z,i)) \ \wedge \\
	\forall y\leq x \ \psi(r+1,s+1,l_{r+1}(s+1,y),z) \ \wedge \\
	\forall i<s \ \forall y\leq x+1 \ \forall k\leq r \ \psi(k,i,l_{k}(i,y),z) 
\end{array}\right).$
\end{center}
Then, for every  $r,s\in I$  , let $E_{r,s}$ be the $\M$-finite subset of $M$ whose $x$-th element is $l_{r}(s,x)$. Moreover, put $\kappa_{r,s}:=\mathrm{Card}(E_{n,s})$. Note that for every $n\in\NN$ and for  each $s\in I$, $l_{n}(s,x)$  is an $\mathbb{L}_{A}$-term. So similar to \cite{sho2} or \cite{k}, we can prove that:
\begin{center}
$(1): \ \ \ $ $\kappa_{n,s}>J$ for every $n\in \NN$ and each $s\in I$.
\end{center} 

Moreover, for every $n\in\NN$, we put:
$$j_{n}:=\max\{j:\ \forall \bar{x} \ \bigwedge_{k=0}^{n} \bm{\phi}_{k}(\bar{x},a)\leftrightarrow\bm{\phi}_{k}(\bar{x},b)\}, $$
$$J:=\inf\{j_{n}: \ n\in\omega\}. $$
By \cite[Lemma 2.7]{k}, we have $I'\subsetneqq J$. As a result:
\begin{center}
$(2): \ \ \ $ There exists some $\textbf{r}\in J\setminus I'$ such that $g(\textbf{r})<\textbf{r}$.
\end{center}  

Now, for every $s\in M$, we define the following recursive type:
$$p_{s}(x):=\{\forall i<s \ \forall z<\textbf{r} \ ((x)_{\langle n,i,z\rangle}=l_{n}(i,z)): \ n\in\omega\}. $$
We need to find some $e>I$ such that $p_{e}(x)$ is finitely satisfiable. For this purpose, define:
$$\Psi(j):=\max\{s: \ \exists x \ \forall k<j \ \forall i<s \ \forall z<\textbf{r} \ (x)_{\langle k,i,z\rangle}=l_{k}(i,z)\}. $$
Let $e>I$ be the witness of the strength of $I$ for  $\Psi$ in $\M$. It is easy to see that $p_{e}(x)$ is finitely satisfiable. As a result,   we can find some $c\in M$ such that $(c)_{\langle n,i,z\rangle}=l_{n}(i,z)$ for all $n\in\NN$, for every $i\in I$, and all $z< \textbf{r}$. 
Now, we define:
$$\Phi(j):=\max\{s: \ \forall i,i'<s \ \forall z<\textbf{r} \ (i\leq i'\leq s\rightarrow \forall k\leq j \ \psi(k,i,(c)_{\langle k,i',z\rangle},(c)_{\langle j,s,\textbf{r}\rangle}))\} .$$
Let $e'>I$ be the witness of the strength of $I$ for $\Phi $ in $\M$. Therefore, by  statement (1)  and by the way we defined $l_{n}(s,x)$, for every $n\in\NN$ it holds that:\\

$(3): \ \ \ $$\M\models\forall k\leq n \ \Phi(k)>e'.$\\

So, by using Overspill principle over $\NN$ in $\M$, from statement $(3)$, we infer that there exists some nonstandard  $n^{*}\in I$ such that:\\

$(4): \ \ \ $$\M\models\forall k\leq n^{*} \ \Phi(k)>e'.$\\

Moreover, let $c':=g(c)$. Again, we define:
$$\Upsilon(j)=:\max\{s: \ \forall i<s \ \forall z<\textbf{r} \ \forall k\leq j \ (c)_{\langle k,i,z\rangle}=(c')_{\langle k,i,z\rangle}\} .$$
Let $e''>I$ be the witness of the strength of $I$ for $\Upsilon $ in $\M$. Therefore, since $\textbf{r}\in J$,  for every $n\in\NN$ it holds that:\\

$(5): \ \ \ $$\M\models\forall k\leq n \ \Upsilon(k)>e''.$\\

So, by using Overspill principle over $\NN$ in $\M$, from statement $(5)$, we infer that there exists some nonstandard  $n^{**}\in I$ such that:\\

$(6): \ \ \ $$\M\models\forall k\leq n^{**} \ \Upsilon(k)>e''.$\\

Therefore, by choosing some nonstandard element $n^{+}\leq\min\{n^{*},n^{**}\}$ such that ${g(n^{+})\leq n^{+}}$, and choosing some $\textbf{s}>I$ such that $\textbf{s}<\Phi(n^{+})$, $\textbf{s}<\Upsilon(n^{+})$, and $g(\textbf{s})\leq\textbf{s}$, we put:
\begin{center}
${\textbf{v}:=(c)_{\langle n^{+},\textbf{s},\textbf{r}\rangle}}$ and $\textbf{u}:=g(\textbf{v})$.
\end{center}  
Statement (2), together with the choice of $\mathbf{s}$ and $n^{+}$, guarantees that $\mathbf{u} < \mathbf{v}$. In order to find the desired  $ \textbf{w}$, we define another recursive type for every $s\in M$:
$$q_{s}(w):=\{\forall i<s \ \varphi(n,i,\textbf{u},\textbf{v},w): \ n\in\omega\}. $$ 
Similar to the previous arguments, by using the strength of $I$ we can find some $s>I$ such that $ q_{s}(w)$ is finitely satisfiable, and
this finishes the proof of Lemma 4.5.1.
\end{proof}
\newtheoremstyle{nonum}{}{}{\itshape}{}{\bfseries}{.}{ }{#1 }
\theoremstyle{nonum}
\newtheorem{lemmm}{Lemma 4.5.2}
\newtheoremstyle{nonum}{}{}{\itshape}{}{\bfseries}{.}{ }{#1 }
\theoremstyle{nonum}
\begin{lemmm}
	Suppose $g,h\in\G_{(I)}$, $\I(h)\subseteq\I(g)$, $a,b\in M$, and $g(a)=b$. Then there exist some $f_{1},f_{2}\in\G_{(I)}$ such that: 
	\begin{center}
		$h^{-f_{1}}h^{f_{2}}(a)=b$ or 	$h^{f_{1}}h^{-f_{2}}(a)=b$.
	\end{center}
\end{lemmm}
\begin{proof}[Sketch of proof]
Lemma 4.5.2 follows from Lemma 4.5.1 in the following way (the following argument is analogous to the argument presented in the proof of \cite[Theorem 4.5]{sho2}): 

Suppose there are arbitrary small $\textbf{x}>I':=\I(h)$ such that $h(\textbf{x})<\textbf{x}$. we will show  that this leads the first case of the lemma (if instead there are arbitrary small $\textbf{x}>I':=\I(h)$ such that $h(\textbf{x})>\textbf{x}$, a similar argument leads to the second case of the lemma). Moreover, let $J'$ be the largest invariant cut in $\G_{(I)}$ which is closed under exponentiation and is contained in $I'$.

First, suppose $a<b$.  If   $I'=J'$, put $h':=h$ and $f:=id$. If $J'\subsetneqq I'$, by part (3) of this theorem
 the closure of $\G_{(>J')} $ in $\G_{(I)}$ is $\G_{(J')}$. Then there exist some $h'\in \G_{(>J')}$  and some $f\in\G_{(I)}$ such that $h'(a)=b$ and $h'^{f}\in\G_{(I')}$. Then, in both cases by applying Lemma 4.5.1 for $h$, $h'^{f} $, $f^{-1}(a)$, and $f^{-1}(b)$, we will find some $\textbf{u},\textbf{v},\textbf{w}\in M$ such that:
\begin{center}
	$ \textbf{u}<\textbf{v}<\textbf{w}$, $h(\textbf{v})=\textbf{u}$, and\\ $\tp(\textbf{u},\textbf{v},i)=\tp(\textbf{u},\textbf{w},i)$ and $\tp(\textbf{v},\textbf{w},i)=\tp(f^{-1}(a),f^{-1}(b),i)$ for all $i\in I$.
\end{center} 

Therefore, there exist some $\sigma,\rho\in\G_{(I)} $ such that $\sigma(\textbf{u},\textbf{v})=(\textbf{u},\textbf{w})$ and $\rho(\textbf{v},\textbf{w})=(f^{-1}(a),f^{-1}(b))$. Then, put $f_{1}:=\sigma^{-1}\rho^{-1}f^{-1}$ and $f_{2}:=\rho^{-1}f^{-1}$.

Moreover, if $b<a$, consider $h$, $h'^{(-f)}$, $f^{-1}(b)$, and $f^{-1}(a)$ in Lemma 4.5.1.
\end{proof}
\end{proof}
Before turning to the proof of Theorem 4.4, let us first consider the following corollary, which follows from Theorem 4.5:
\begin{cor}
suppose $\M$ and $\N$ are countable recursively saturated models of $\pa$ sharing a  cut $I$ which is strong in both and $I\prec\M$. Moreover, let $(\A(\M))_{(I)}$ be isomorphic to $((\A(\N))_{(I)}$---as topological groups. Then $I\prec\N$. In particular, we have:
\begin{itemize}
\item[(1)] The assumption of being isomorphic of $(\A(\M))_{(I)}$ and $(\A(\N))_{(I)}$ can be reduced to being isomorphic as two groups. 
\item[(2)] If $\ssy(\M)=\ssy(\N)$, then $\M$ is isomorphic to $\N$.
\end{itemize}
\end{cor}
\begin{proof}
Suppose not; i.e $I$ is not an elementary submodel of $\N$. So, $I\subsetneqq\KK(\N;I)$. Thus, by Theorem 4.5, there exists some nontrivial closed normal subgroups of  $(\A(\N))_{(I)}$. But since $(\A(\N))_{(I)}\cong(\A(\M))_{(I)} $ and $I\prec\M$ this leads to a contradiction. Part (1) of the theorem is inferred from Theorem A, and part (2) is an implication of Theorem 2.1.
\end{proof} 

Now, we are ready to see the proof of Theorem 4.4:
\begin{proof}[Proof of Theorem 4.4]
	If $\emptyset\neq\mathrm{N}$ is open in $\G_{(I)}$, then it is also closed in $\G_{(I)}$. Then by Theorem 4.5, $\mathrm{N}=\G_{(\I(\mathrm{N}))}$.  As a result, $\G_{(\I(\mathrm{N}))} $ is open in $\G_{(I)}$, which implies that  $\G_{(I)}\subseteq\mathrm{N}$.

If $\mathrm{N}$ is not open in $\G_{(I)}$, then we will show that $\mathrm{N}\cap\LG_{I}=\emptyset$. Suppose not; i.e., there exists some $g\in\mathrm{N}\cap\LG_{I}$. Let $\K$ be some $I$-small elementary submodel of $\M$ such that $g\upharpoonright_{K}$ is $I$-e.c.  Then, by Theorem 3.6 and Lemma 2.8  there exists some Lascar $I$-generic automorphism $h\in \G_{(K)}\setminus\mathrm{N}$. But this is in contradiction with Theorem 3.5 and normality of $\mathrm{N}$.
\end{proof} 
\begin{q}
	Suppose $\M$ is a countable recursively saturated model of $\pa$, $\G:=\A(\M)$, and $I$ is a strong cut of $\M$. Moreover,	let $g\in\G_{(I)}\setminus\LG_{I}$ such that $\I(g)=I$ (such $g$ exists by Corollary 4.2). What is the normal subgroup generated by its conjugacy class in $\G$? 
\end{q}
 The next corollary concerns the homomorphic images of $(\A(\M))_{(I)}$; a result that will be important in investigating the action of $(\A(\M))_{(I)}$ on trees (see Question 4 at the end of the paper for more details).
\begin{cor}
	Suppose $\M$ is a countable and recursively saturated model of $\pa$, ${\G:=\A(\M)}$, and $I$ is a strong cut of $\M$. Then the infinite cyclic group $\mathbb{Z}$ is not a homomorphic image of $\G_{(I)}$.
\end{cor}
\begin{proof}
	Suppose not; i.e. there exists some homomorphism $\Phi$ such from $\G_{(I)}$ onto $\mathbb{Z}$. So ${[\G_{(I)}:\mathrm{ker}(\Phi)]=|\mathbb{Z}|}$.   As a result,  by Lemma 3.1, $\mathrm{ker}(\Phi) $ is not meagre in $\G_{(I)}$. Thus, from Theorem 4.4 we infer that $\mathrm{ker}(\Phi)=\G_{(I)}$, which  contradicts the assumption that $\Phi$ is onto.
\end{proof}

\section{The cofinality of $(\A(\M))_{(I)}$}
As mentioned in the Introduction, the cofinality of a group $\G$, denoted by $\mathrm{cf}(\G)$, is the least cardinal number $\kappa$ such that $\G$ is the union of an increasing chain of $\kappa$ many proper subgroups. In \cite{ks}, Kossak and Schmerl proved that \textit{for a countable recursively saturated model $\M$ of $\pa$ the standard cut is strong in $\M$ iff $\mathrm{cf}(\A(\M))>\aleph_{0}$.} In this this section, using similar arguments,  we generalize their result  to a given cut which is not $\omega$-coded from above:

Suppose $\M$ is a  model of $\pa$ and $I$ is a cut of $\M$. We say  $I$  is \textit{$\omega$-coded from above} if there  exists some $a\in M$ such that ${I=\inf\{(a)_{n}: \ n\in\omega\}}$. 
	It is easy to see that if $\M$ is recursively saturated and $I$ is not $\omega$-coded from above in $\M$, then for every $a\in M\setminus\mathrm{K}(\M;I)$, there exists some $b\in M $ such that $a\neq b$ and $\tp(a,i)=\tp(b,i)$ for all $i\in I$. In particular, $(\A(\M))_{(I)}\neq(\A(\M))_{(I\cup\{a\})} $. 

\begin{lem}
Suppose $\M$ is a countable and recursively saturated model of $\pa$, ${\G:=\A(\M)}$, and $I$ is a cut of $\M$ which is not $\omega$-coded from above. Then $I$ is strong in $\M$ iff for every open subgroup  $\mathrm{H}$ of $\G_{(I)}$ there exists some $g\in \G_{(I)}$ such that $\G_{(I)}=\langle \mathrm{H}\cup\{g\}\rangle$.
\end{lem}
\begin{proof}
 Let $S$ be an inductive satisfaction class for $\M$ such that $(\M;S)$ is also recursively saturated. First, suppose $I$ is strong in $\M$, and $\mathrm{H}$ is an open subgroup  of $\G_{(I)}$. Thus, there exists some $a\in M$ such that $\G_{(I\cup\{a\})}\leq \mathrm{H}$. It suffices to find some $b\in M$ such that $\tp(a,i)=\tp(b,i)$ for all $i\in I$ and $\G_{(I)}=\langle\G_{(I\cup\{a\})}\cup\G_{(I\cup\{b\})}\rangle$ (since then we can find some  $g\in \G_{(I)}$ such that $g(b)=a$. So $g^{-1}\G_{(I\cup\{a\})}g=\G_{(I\cup\{b\})}$, which implies that  $\G_{(I)}=\langle \mathrm{H}\cup\{g\}\rangle$).\\
  In order to find such $b$, first let $\delta\in M$ such that 
$\mathrm{K}(\M;I)=\{(\delta)_{i}: \ i\in I\}$. Then, we
define:
$$\Phi( \langle r,i\rangle):=\max\{s: \ \forall j<s \  S(\bm{\phi}_{r}(a,(\delta)_{j},i))\}. $$
By using the strength of $I$ in $\M$, there exists some $e>I$ such that   $\Phi(\langle r,i\rangle)\in I$ iff $\langle r,i\rangle<e$ for all $\langle r,i\rangle\in I$. Then, for every $s\in M$ let:
\begin{align*}
p_{s}(x):=&\{\forall i<s \ \phi(a,i)\leftrightarrow\phi(x,i): \ \phi \text{ is an }\mathbb{L}_{A}\text{-formula}\}\cup\\&
\{ \forall i<s \ \Phi(\langle\ulcorner\phi\urcorner,i\rangle)>e\rightarrow   \phi(a,x,i)):\ \phi \text{ is an }\mathbb{L}_{A}\text{-formula}\}.
\end{align*}
 Our aim is to find some $s>I$ such that $p_{s}(x)$ is finitely satisfiable. For this purpose, we define:
$$\Psi(r):=\max\left\lbrace s: \ \exists x  \left(\begin{array}{c}
\forall k<r \ \forall i<s \ S(\bm{\phi}_{r}(a,i))\leftrightarrow S(\bm{\phi}_{r}(x,i)))\ \wedge \\ \forall k<r \ \forall i<s \  (\Phi(\langle r,i\rangle)>e\rightarrow S(\bm{\phi}_{r}(a,x,i)))
\end{array}\right)\right\rbrace . $$
Again, let $s_{0}>I$ which witnesses the strength of $I$ in $\M$ for $\Psi$. Therefore, $p_{s_{0}}(x)$ is finitely satisfiable: since otherwise, there exists some  $n\in\omega$ such that $\Psi(n)<s_{0}$. Let $i_{0}:=\Psi(n)+1\in I$. So it holds that:\\

$(1): \ \ \ \M\models\forall x 
  \left(\begin{array}{c} \bigwedge_{k=0}^{n} \forall i<i_{0} (\bm{\phi}_{n}(a,i)\leftrightarrow\bm{\phi}_{n}(x,i))\rightarrow \\
\bigvee_{k= 0}^{n}\exists i<i_{0} \ (\Phi(k,i)>e \ \wedge \ \neg\bm{\phi}_{k}(a,x,i))
\end{array}\right)$.\\

Then, let $\alpha\in I$ be the code of the set $\{\langle k,i\rangle\in I: \ \M\models k\leq n \ \wedge \ i<i_{0} \ \wedge \ \bm{\phi}_{n}(a,i)\}$. Moreover, let $\beta:=\mu_{x} \bigwedge_{k=0}^{n} (\forall i<i_{0} (\langle k,i\rangle\E\alpha\leftrightarrow\bm{\phi}_{n}(x,i))\in \mathrm{K}(\M;I)$.  As a result, statement (1) implies that:\\

$(2): \ \ \ \M\models 
\bigvee_{k=0}^{n}  
	\exists i<i_{0} \ (\Phi(k,i)>e \ \wedge \ \neg\bm{\phi}_{k}(a,\beta,i))
$.\\

But statement (2) is in contradiction by the definition of $\Phi$.

Therefore, let $b\in M$ realizes $p_{s_{0}}(x)$. In order to see $\G_{(I)}=\langle\G_{(I\cup\{a\})}\cup\G_{(I\cup\{b\})}\rangle$, let $f\in\G_{(I)}$ be arbitrary and $f(a)=c$. Again, for every $s\in M$, we define:
$$q_{s}(x):=\{\forall i<s \ \phi(a,b,i)\leftrightarrow\phi(a,x,i)\leftrightarrow\phi(c,x,i): \ \phi\text{ is an }\mathbb{L}_{A}\text{-formula}\}. $$
Then, put:
$$\Omega(r):=\max\{s: \ \exists x \ \forall k<r \ \forall i<s \ (S(\bm{\phi}_{k}(a,b,i))\leftrightarrow S(\bm{\phi}_{k}(a,x,i))\leftrightarrow S(\bm{\phi}_{k}(c,x,i)))\} .$$
Again, let $s_{1}>I$ be the witness of strength of $I$ in $\M$ for $\Omega$. So $p_{s_{1}}(x)$ is finitely satisfiable: since if $\Omega(n)<s_{1}$ for some $n\in \omega$, then let $i_{1}:=\Omega(n)+1\in I$. As a result, it holds that:\\

$(3): \ \ \ \M\models\forall x \ \exists i<i_{1} \ \neg\bigwedge_{k=0}^{n}(\bm{\phi}_{k}(a,b,i)\leftrightarrow\bm{\phi}_{k}(a,x,i)\leftrightarrow\bm{\phi}_{k}(c,x,i))$.\\

Then, let $\nu\in I$ be the code of the set $\{\langle k,i\rangle\in I: \ \M\models k\leq n \ \wedge \ i<i_{1} \ \wedge \ \bm{\phi}_{n}(a,b,i)\}$.   Since $\tp(a,i)=\tp(c,i)$ for all $i\in I$, from statement (3) we infer that for every $\textbf{x}\in K(\M;I)$ it holds that:

$(4): \ \ \ \M\models \overset{\varphi(a,\textbf{x},i_{1},\nu,)}{\overbrace{\exists i<i_{1}  \neg\bigwedge_{k=0}^{n}(\langle k,i\rangle\E\nu\leftrightarrow\bm{\phi}_{k}(a,\textbf{x},i))}}$.\\

Thus,  by the way we chose $b$, statement (4) implies that $\M\models\varphi(a,b,i_{1},\nu)$, which is a contradiction.

 Finally, let $d\in M$ realizes $p_{s_{1}}(x)$ in $\M$, and $h,l\in\G_{(I)}$ such that $h(a,d)=(c,d)$ and $l(a,b)=(a,d)$. So we have: $$f=hh^{-1}f=\underset{\in \G_{(I\cup\{a\})}}{\underbrace{l}}\overset{\in \G_{(I\cup\{b\})}}{\overbrace{l^{-1}hl}}\underset{\in \G_{(I\cup\{a\})}}{\underbrace{l^{-1}}}\overset{\in \G_{(I\cup\{a\})}}{\overbrace{h^{-1}f}}.$$

For proving the right-to-left direction of the lemma, we assume that $I$ is not strong in $\M$; i.e. there exists some $a\in M$ such that $\{(a)_{i}:\ i\in I\}$ is downward cofinal in  ${M\setminus I}$. Without loss of generality, we can assume that $\{(a)_{i}:\ i\in I\}\cap (M\setminus\mathrm{K}(\M;I)) $ is downward cofinal in $M\setminus I$ (since otherwise, we can consider $\{(a')_{i}: \ i\in I\}$, where ${(a')_{i}:=\min(M\setminus\{\textbf{t}_{r}(j)\leq(a)_{i}: \ \langle r,j\rangle<(a)_{i})\}})$). Now, by the  assumption on the  right side of the lemma, there exists some $g\in\G_{(I)}$ such that $\G_{(I)}=\langle\G_{(I\cup\{a\})}\cup\{g\}\rangle$. Let $g(a)=b$. As a result, for every $s\in I$ it holds that:\\

$ (\ast):\ \ \  \M\models\forall i<s \ ((a)_{i}<s\rightarrow (a)_{i}=(b)_{i}) $.\\

So, by using overspill over $I$ in statement $(\ast)$, there exists some $\textbf{s}>I$ such that:\\

$(\ast\ast): \ \ \ \M\models\forall i<\textbf{s} \ ((a)_{i}<\textbf{s}\rightarrow (a)_{i}=(b)_{i}) $.\\

  Now, let  $i_{0}\in I$ such that $I<(a)_{i_{0}}<\textbf{s}$ and $(a)_{i_{0}}\notin\mathrm{K}(\M;I)$. So by $(\ast\ast)$, we have $\G_{(I)}=\langle\G_{(I\cup\{a\})}\cup\{g\}\rangle\leq\G_{(I\cup\{(a)_{i_{0}}\})}$, which is a contradiction (since $I$ is not $\omega$-coded from above).
\end{proof}
\begin{rem}
We can refine the proof of the right-to-left direction of the lemma to obtain the following generalization: \textit{If for every open subgroup $\mathrm{H}$ of $\G_{(I)}$ there exists a finite number of elements of $\G_{(I)}$, say $g_{1},...,g_{n}$ such that $\G_{(I)}=\langle\mathrm{H}\cup\{g_{1},...,g_{n}\}\rangle$, then $I$ is strong in $\M$.} 
\end{rem}
\begin{theorem}[or \textbf{Theorem C}]
Suppose $\M$ is a countable and recursively saturated model of $\pa$, $\G:=\A(\M)$, and $I$ is a cut of $\M$  which is not $\omega$-coded from above. Then $I$ is strong in $\M$ iff $\mathrm{cf}((\A(\M))_{(I)})>\aleph_{0}$.
\end{theorem}
\begin{proof}
	First suppose $I $ is strong in $\M$ and $\G_{(I)}=\bigcup_{n\in\omega}\mathrm{H}_{n}$ such that $\{\mathrm{H}_{n}: \ n\in\omega\}$ is an increasing sequence of proper subgroups of $\G_{(I)}$. So by the previous lemma, none of $\mathrm{H}_{n}$s is open in $\G_{(I)}$. Now, we will inductively build sequences $\langle \M_{s}: \ s\in 2^{<\omega}\rangle$ of  $I$-small submodels of $\M$ and $\langle h_{s}: \ s\in 2^{<\omega} \rangle$ of elements of $\G_{(I)}$ such that:
	\begin{itemize}
	\item[(1)] $M=\bigcup_{n\in\omega} M_{\sigma\upharpoonright_{n}}$ for every $\sigma\in 2^{\omega}$.
	\item[(2)] For every $s,t\in 2^{<\omega}$ if $s\sqsubset_{e}t$, then $M_{s}\subseteq M_{t}$.
	\item[(3)] For every $s,t,t'\in 2^{<\omega}$ if $s\sqsubset_{e}t$ and $s\sqsubset_{e}t'$, $h_{t}\upharpoonright_{M_{s}}=h_{t'}\upharpoonright_{M_{s}}$.
	\item[(4)] For every $n\in\omega$ and for all $s\in 2^{<\omega}$, if $\mathrm{len}(s)=n$ then $h_{s\frown 1}\notin h_{s\frown 0}\mathrm{H}_{n}$.
	\end{itemize}
	For constructing these sequences, we first fix some enumeration $\{a_{n}: \ n\in\omega\}$ of $M$. For the first step of the induction, let $\M_{\emptyset}$ be an $I$-small submodel of $\M$ containing $a_{0}$, and put $h_{\emptyset}=h_{0}=id$. Since $\mathrm{H}_{0}$ is not open in $\G_{(I)}$, then there exists some $h_{1}\in h_{0}\G_{(M_{\emptyset})}\setminus h_{0}\mathrm{H}_{0}$. So this finishes the first step of induction. Suppose $n\in\omega$ is given and we have built the sequences for  $s\in 2^{<\omega} $ such that $\mathrm{len}(s)= n$. First, let $\M_{s\frown 0}=\M_{s\frown 1}$ be an $I $-small submodel of $M$ containing $a_{n+1}$ and $M_{s}$, and let $h_{s\frown 0}=h_{s}$. Again, since $\mathrm{H}_{n}$ is not open in $\G_{(I)}$, there exists some $h_{s\frown 1}\in h_{s}\G_{(M_{s})}\setminus h_{s\frown 0}\mathrm{H}_{n}$. So we are done with the instruction of required sequences.

	Now, for every $\sigma\in 2^{\omega}$, let $h_{\sigma}:=\lim\limits_{n\rightarrow\infty} h_{\sigma\upharpoonright_{n}}	$ (by conditions (1) and (2) of our construction this limit exists). Moreover, by conditions (3) and (4), $h_{\sigma}\neq h_{\rho}$ for every distinct $\sigma,\rho\in 2^{\omega}$. As a result, since $\G_{(I)}=\bigcup_{n\in\omega}\mathrm{H_{n}}$, there exists some $n\in \omega$ which contains uncountably many of $h_{\sigma}$s where $\sigma\in 2^{\omega}$. Therefore, there exist some $ \sigma,\rho\in 2^{\omega}$ and some $m\geq n$ such that $h_{\sigma},h_{\rho}\in \mathrm{H}_{n}$ $\sigma\upharpoonright_{m}=\rho\upharpoonright_{m}$ and $\sigma(m)\neq\rho(m)$. However,  by conditions (3) and (4), it holds that $h_{\sigma}\notin h_{\rho}\mathrm{H}_{m}$ or $h_{\rho}\notin h_{\sigma}\mathrm{H}_{m}$, which is a contradiction (since $\mathrm{H}_{n}\leq\mathrm{H}_{m}$).

	For proving the other direction, suppose $\mathrm{cf}((\A(\M))_{(I)})>\aleph_{0}$. Let $\mathrm{H}$ be an open subset of $\G_{(I)}$. By Remark 2, it suffices to prove that  $\G_{(I)}$ is finitely generated over $\mathrm{H}$. Since $\mathrm{H}$ is open in $\G_{(I)}$, we have $[\G_{(I)}:\mathrm{H}]\leq\aleph_{0}$ (since $\G_{(I)}$ is Polish).  Let  $\{g_{n}: \ n\in \omega\}$ be a sequence of elements of $\G_{(I)}$ such that $\G_{(I)}=\bigcup_{n\in\omega}g_{n}\mathrm{H}$. As a result, $\G_{(I)}=\bigcup_{n\in\omega}\langle \mathrm{H}\cup\{g_{0},...,g_{n}\}\rangle$. Therefore, since $\mathrm{cf}((\A(\M))_{(I)})>\aleph_{0}$, there exists some $ n\in \omega$ such that $\G_{(I)}=\langle \mathrm{H}\cup\{g_{0},...,g_{n}\}\rangle$. 
\end{proof}

\begin{cor}
Suppose $\M$ and $\N$  are countable and recursively saturated model of $\pa$,  $I$ is strong cut of $\M$, and  $J\subseteq_{e}\N$ which is not $\omega$-coded from above. If   $(\A(\M))_{(I)}\cong(\A(\N))_{(J)} $ (as groups), then $J$ is strong in $\N$. Moreover, if $I\prec_{e}\M$, then $J\prec_{e}\N$.
\end{cor}
\begin{proof}
The first part of the corollary is a direct consequence of Theorem 5.2. The second part is deduced from Corollary 4.6.
\end{proof}

We conclude the paper with the following questions:
\begin{q}
As mentioned in the Introduction, Lascar introduced his notion of generic automorphisms based on Truss's notion of a generic automorphism; that is, an automorphism whose class of conjugacy is comeagre. More generally, several results established in this paper, such as Theorems A and B, hold when $\G$ is a Polish group possessing a comeagre conjugacy class. 

Now, suppose $\M$ is a countable recursively saturated model of $\pa$, $\G:=\A(\M)$, and $I$ is a cut of $\M$.  We ask under what conditions the subgroup $\G_{(I)}$ has an element $g$ such that $[g]^{\G_{(I)}}$ is comeagre in $\G_{(I)}$. In \cite{sch}, Schmerl proved the following result in the case $I$ is equal to the standard cut:
\begin{theorem}[Schmerl \cite{sch}]
Suppose $\M$ is a countable recursively saturated model of $\pa$ in which $\NN$ is strong. 
\begin{itemize}
\item[(1)] If $\M\models\mathrm{Th}(\NN)$, then there exists an automorphism of $\M$ whose conjugacy class is comeagre in $\A(\M)$.
\item[(2)] If $\M\nvDash\mathrm{Th}(\NN)$, then $\M$ has an automorphism  whose conjugacy class is comeagre in $\A(\M)$  iff $\M\models\mathrm{H}(n,4)$ for some $n\in\NN$ (where in general, $ \mathrm{H}(x,y)$ is the statement stating "\textit{For any two digraphs $D_{1}$ and $D_{2}$ who have an $x$-coloring, their product digraph $D_{1}\times D_{2}$ has a $y$-coloring}"). 
\end{itemize}
\end{theorem}
Now, it is natural to investigate whether the same results hold for nonstandard strong cuts. In his prove, Schmerl makes essential use of a result by A. Hajnal which appears in \cite{haj}\footnote{Hajnal's result states that \textit{If $D_{1}$ and $D_{2}$ are two digraphs such that $\chi(D_{1}),\chi(D_{2})>n $ for all $n\in\NN$, then $\chi(D_{1}\times D_{2})$ is also larger than  $n $ for all $n\in\NN$}.}. Understanding how much of this result can be recovered for nonstandard cuts may allow one to adapt the techniques developed in this paper for strong cuts and thereby generalize Theorem 5.4 above.
\end{q}

\begin{q}
In \cite{ser},  J. P. Serre's introduced  property $(\mathrm{FA})$ as follows: we say that a group $\G$ \textit{acts without inversion} on a tree $T$ if $\G\leq\mathrm{Aut}(T)$ (as an abstract group) and for no $g\in\G $ there do not exist two adjacent vertices $t,s\in T$ such that $g.t=s$ and $g.s=t$. $\G$ is said to have \textit{property $(\mathrm{FA})$} if whenever $\G$ acts without inversion on a tree $T$, then  there exists some $t\in T$ such that $g.t=t$ for all $g\in \G$. \\
Serre in \cite{ser} showed  that \textit{any group  $\G$ has property $\mathrm{FA}$ iff the following conditions hold: (1) the infinite cyclic group $\mathbb{Z}$ is not a homomorphic image of $\G$, (2) the cofinality of $\G$ is uncountable, and (3) $\G$ is not a non-trivial free product with amalgamation; that is, whenever $\G$ is a free product with amalgamation, say $\G=\G_{1}\ast_{A}\G_{2}$, then $\G=\G_{i}$ for some $i=1,2$}.

In view of Corollary 4.7 and Theorem C, conditions (1) and (2) above hold for $(\A(\M))_{(I)}$, when $I$ is a strong cut of the countable recursively saturated model $\M$ of $\pa$. This raises the question of whether property  $(\mathrm{FA})$ also holds for $(\A(\M))_{(I)} $.

\end{q}

\end{document}